\documentclass[a4paper,12pt]{amsart}

\usepackage{a4wide}
\usepackage{amsmath,amsthm,latexsym,amssymb,pgf,enumitem,hyperref} 
\usepackage{tikz,subcaption,graphicx} 
\usepackage{multirow}
\usetikzlibrary{arrows,decorations.pathmorphing,backgrounds,positioning,fit,petri,calc}

\newtheorem{lemma}{Lemma}

\newtheorem{theorem}[lemma]{Theorem}

\theoremstyle{remark}

\usepackage[notref,notcite, final]{showkeys}

\title{Minimum Wiener Index of Triangulations and Quadrangulations}  

\author{\'Eva Czabarka}
\author{Trevor V. Olsen}
\author{Stephen J. Smith}
\author{L\'aszl\'o Sz\'ekely}
\address{\'Eva Czabarka\\Department of Mathematics \\ University of South Carolina \\ Columbia SC 29212 \\ USA}
\email{czabarka@math.sc.edu}
\address{Trevor V. Olsen\\Department of Mathematics\\ University of South Carolina \\ Columbia SC 29212 \\ USA}
\email{tvolsen@email.sc.edu}
\address{Stephen J. Smith\\Department of Mathematics\\ University of South Carolina \\ Columbia SC 29212 \\ USA}
\email{sjs8@email.sc.edu}
\address{L\'aszl\'o Sz\'ekely\\Department of Mathematics \\ University of South Carolina \\ Columbia SC 29212 \\ USA}
\email{szekely@math.sc.edu}

\subjclass[2010]{Primary: 05C12 Secondary: 05C10}
\keywords{ planar graph, triangulation, quadrangulation, distance, Wiener index, connectivity, 
maximum degree.}              

\begin{document}

\begin{abstract}
 The Wiener index of a connected graph is the sum of the distances between all unordered pairs of vertices. We provide formulae for the minimum Wiener index of simple triangulations and quadrangulations with  connectivity at least $c$,  and provide the extremal structures, which attain those values. Our main tool is setting upper bounds for the maximum degree in highly connected triangulations and quadrangulations.
\end{abstract}

\maketitle


\section{Definitions}
	
Let $G$ be a connected graph. The {\sl Wiener index} of $G$, denoted by $W(G)$, is the sum of the distances between all unordered pairs of vertices. In formula,  
\[ W(G) = \sum_{\{u,v\} \subseteq V(G)} d_G(u,v), \]
where $d_G(u,v)$ denotes the number of edges on the shortest path between the two vertices $u$ and $v$. This index was introduced in 1947  \cite{Wie1947}  to predict the boiling point of alkanes. The Wiener index is perhaps 
the most frequently used graph parameter in the sciences.

Throughout this paper, every graph will be  simple, finite and connected unless otherwise stated. For a graph $G$, the sets $V$ and $E$ represent the vertices and edges of $G$, respectively.  The order of a graph is its number of vertices. The set of {\sl neighbors} of the vertex $v$ is denoted by $N(v)$, and 
the {\sl degree} of a vertex $v$ is denoted by $d(v)=|N(v)|$. 
We denote by $\delta(G)$ and $\Delta(G)$ the minimum and maximum degree of $G$, respectively.
A  {\sl cutset} is a set of vertices, whose removal makes the graph disconnected.
A non-complete graph $G$ of order $n\ge 3$  is {\sl $c$-connected } for a positive integer $c$, if every cutset has size at least $c$; the {\sl connectivity} $\kappa(G)$ of $G$
is the largest $c$
for which $G$ is $c$-connected. Clearly, if $G$ is not a complete graph, then
$G$ has at least $\kappa(G)+2$ vertices, the smallest cutset of $G$ has size $\kappa(G)$, and all degrees in $G$ are at least $\kappa(G)$. The notation $\simeq$ indicates 
the isomorphism of two graphs.

In this paper  we will only be concerned with {\sl planar} graphs.  Those are the graphs that can be drawn in the plane (or equivalently, in the sphere), such that no edges cross.  We will often rely on {\sl Euler's formula}, which states that  for any finite, connected planar graph $G$ drawn in the plane, $$n -e +f =2,$$ where $n$ is the order, $e$ is the number of edges, and $f$ is the number of faces in $G$. In this paper, {\sl triangulations} and {\sl quadrangulations} are simple graphs drawn in the sphere, in which every face is a triangle
or every face is a quadrangle, respectively. 
Euler's formula immediately implies that triangulations of order  $n$  have $3n-6$ edges and $2n-4$ faces, and quadrangulations with $n$ vertices have $2n-4$ edges and $n-2$ faces. It is well-known that triangulations are $3$-connected but Euler's formula does not allow them to be $6$-connected, and  quadrangulations are
$2$-connected but  Euler's formula does not allow them to be $4$-connected.  Whitney's theorem \cite{whitney} implies that all drawings of  any $3$-connected planar graph on the 
sphere are the same combinatorially, a conclusion that holds for all  classes that we consider in this paper except general quadrangulations.
We cite an elegant result of \upshape{\cite{NogSuz2015}}, although we do not use it explicitly:
every $5$-connected triangulation contains a spanning 
$3$-connected quadrangulation. Comparison of Figures   \ref{fig:evenq3cminwi} and \ref{fig:t5cminwi} incidentally  gives an
illustration for this result.

\section{ Results on  Triangulations and Quadrangulations}

Recently, there have been numerous results regarding the Wiener index on {\sl triangulations} and {\sl quadrangulations} of the sphere, which are edge maximal simple planar graphs, and edge maximal bipartite simple planar graphs, respectively. These recent results have mainly focused on upper bounds, see \cite{CheCol2019}, \cite{CzaDanOlsSze2019}, \cite{GhoGyoPauSalZam2019}, \cite{GyoPauXia2020}. Lower bounds for the Wiener index of such graphs were stated in \cite{CheCol2019}, \cite{CzaDanOlsSze2019}, without 
making extra 
assumptions on the connectivity. In this paper we complete the study of the minimum Wiener index of  triangulations and quadrangulations, by determining the minimum Wiener index among  $c$-connected simple triangulations and quadrangulations.

\begin{theorem}  [\cite{CzaDanOlsSze2019}]\label{theo:tminwi} Assume $n\geq 6$.
    The triangulation $ T_n^4 $ defined in {\upshape Figure ~\ref{fig:t4c}} minimizes the Wiener index among all  triangulations of order $n$.  The triangulation $ T_n^4 $ is $4$-connected. 
    Consequently, the triangulation $ T_n^4 $ minimizes the Wiener index  among all $4$-connected $n$-vertex triangulations as well.
\end{theorem}
\noindent {\bf Remark:} $ T_5^4 $ is the only triangulation of order $5$, but it is not  $4$-connected.  Gray vertices and dashed edges in the figures indicate the pattern to be repeated as $n$ increases.
\begin{proof}
    A triangulation contains $ 3n-6 $ edges, thus there are exactly $ 3n-6 $ pairs of vertices at distance 1 apart. If we can make sure that every remaining pair of vertices are at distance 2 apart, then we have a triangulation whose Wiener index is  $ 2(\binom{n}{2} - (3n-6)) + (3n-6) = n^2 -4n +6 $, and this is clearly the minimum possible Wiener index. This is the case with $ T_n^4 $. Furthermore,
    it is easy to see that  $ T_n^4 $ is $4$-connected for all $n\ge 6$. 
    \end{proof}

		\begin{figure}[htbp]
		\centering
			\begin{tikzpicture}
			[scale=0.7,inner sep=0.75mm, 
			vertex/.style={circle,draw}, 
			thickedge/.style={line width=0.75pt}] 
			\node[vertex] (p1) at (2.5,4) [fill=black] {};
			\node[vertex] (p2) at (2.5,2) [fill=black] {};
			
			\node[vertex] (a1) at (0,3) [fill=black] {};
			\node[vertex] (a2) at (1,3) [fill=black] {};
			\node[vertex] (a3) at (2,3) [fill=black] {};
			\node[vertex] (a4) at (3,3) [fill=black] {};
			\node[vertex,dashed] (a5) at (4,3) [fill=gray] {};
			\node[vertex] (a6) at (5,3) [fill=black] {};
			
			\draw[thickedge] (p1)--(a1);
			\draw[thickedge] (p1)--(a2);
			\draw[thickedge] (p1)--(a3);
			\draw[thickedge] (p1)--(a4);
			\draw[thickedge,dashed] (p1)--(a5);
			\draw[thickedge] (a1)--(a2)--(a3)--(a4);
			\draw[thickedge,dashed] (a4)--(a5);
			\draw[thickedge] (a5)--(a6);
			\draw[thickedge] (p1)--(a6);
			\draw[thickedge] (a1) to[out=90,in=180] (2.5,4.5);
			\draw[thickedge] (a6) to[out=90,in=0] (2.5,4.5);
			
			\draw[thickedge] (p2)--(a1);
			\draw[thickedge] (p2)--(a2);
			\draw[thickedge] (p2)--(a3);
			\draw[thickedge] (p2)--(a4);
			\draw[thickedge,dashed] (p2)--(a5);
			\draw[thickedge] (p2)--(a6);
			
		\end{tikzpicture}
			\caption{The triangulation $T_n^4$, which is the join of  the cycle $C_{n-2}$ with the edgeless graph on two vertices, minimizes the Wiener index among all triangulations of order $n\ge 5$ and are $4$-connected for $n\ge 6$.} 	
		\label{fig:t4c}
	\end{figure}
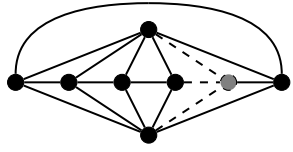
	
Triangulations and $4$-connected triangulations  fail to produce unique structures to minimize the Wiener index. With the aid of a computer, we evaluated the number of non-isomorphic
triangulations  of minimum Wiener index  up to order $18$, and  the number of non-isomorphic $4$-connected triangulations  of minimum Wiener index up to order $22$,  see Table \ref{tab:tsummary}. 
Considering the large numbers in the Table \ref{tab:tsummary}, the classification
of extremal structures seems hopeless. As will be shown, there are two $5$-connected triangulations on $19$ vertices that minimize the Wiener index.
All other graph classes studied in this paper will produce unique extremal graphs minimizing the Wiener index. 

\begin{table}
    \centering
		\begin{tabular}{ |c|c|c| } 
			\hline
			\multirow{2}{*}{Order} & General & 4-Connected   \\ [0.5ex] 
			& Triangulation Count & Triangulation Count \\
			\hline
			4 & 1 & 0 \\
			\hline
			5 & 1 & 0 \\ 
			\hline
			6 & 2 & 1 \\ 
			\hline
			7 & 5 & 1 \\ 
			\hline
			8 &12 & 2 \\ 
			\hline
			9 & 36 & 4 \\ 
			\hline
			10 & 99 & 6 \\ 
			\hline
			11 & 255 & 10 \\ 
			\hline
			12 & 614 & 10 \\ 
			\hline
			13 & 1532 & 14 \\ 
			\hline
			14 & 3908 & 15 \\ 
			\hline
			15 & 10727 & 19 \\ 
			\hline
			16 & 31242 & 21 \\ 
			\hline
			17 & 96725 & 25 \\ 
			\hline
			18 & 311735 & 27 \\
			\hline
			19 &  & 32 \\
			\hline
			20 & & 34 \\
			\hline
			21 & & 39 \\
			\hline
			22 &  & 42 \\
			\hline
		\end{tabular}
		\caption{A summary of how many isomorphism classes, on $n$ vertices, attain the minimum Wiener index 
		for general and $4$-connected triangulations.}
		\label{tab:tsummary}
		
	\end{table}

\begin{theorem}[\cite{CheCol2019},\cite{CzaDanOlsSze2019}] \label{theo:qminwi1}  Assume $n\geq 4$. 
    The complete bipartite graph $ K_{2,n-2} $ 
    minimizes the Wiener index among all quadrangulations. 
\end{theorem}

\begin{proof}
    A quadrangulation contains $ 2n-4 $ edges, thus exactly $ 2n-4 $ pairs of vertices are at distance 1 apart. If we can make sure that 
    every remaining pair of vertices are at distance 2 apart, we have a quadrangulation of Wiener index $ 2(\binom{n}{2} - (2n-4)) + (2n-4) = n^2 -3n +4 $.
    This is the case with the quadrangulation $ K_{2,n-2} $. Clearly this is the least possible Wiener index of a quadrangulation. 
 \end{proof}  
   
   \begin{theorem} \label{theo:qminwiunique}  Assume $n\geq 4$. Up to isomorphism, 
    the graph $ K_{2,n-2} $ 
    is the unique minimizer  of the Wiener index among all quadrangulations of order $n$.
     \end{theorem}
\begin{proof} 
Let $Q$ be a quadrangulation of order $n$ that has the same Wiener index as $K_{2,n-2}$, i.e. every non-adjacent pair of vertices are at distance $2$.
As quadrangulations
are $2$-connected, the minimum degree $\delta:=\delta(Q)\ge 2$. 
Let $v$ be a vertex of $Q$ with $d(v)=\delta$, and let $u_1,\ldots,u_{\delta}$ be the neighbors of $v$. The remaining $n-\delta-1$ vertices 
are at distance $2$ from $v$. As quadrangulations are bipartite, these $n-\delta-1$ vertices can only be adjacent to $u_1,\ldots,u_{\delta}$, and have degree at least $\delta$. 
Thus we get that
$Q\simeq K_{\delta,n-\delta}$. Since $\delta$ is the minimum degree, $\delta\le n-\delta$, therefore $Q$ contains $K_{\delta,\delta}$ as a subgraph. Since $Q$ is planar, we get $\delta=2$ and $Q\simeq K_{2,n-2}$.
\end{proof}

\begin{figure}[htbp]
	\centering
		\begin{tikzpicture}
		[scale=0.8,inner sep=0.75mm, 
		vertex/.style={circle,draw}, 
		thickedge/.style={line width=0.75pt}] 
		\node[vertex] (p1) at (2.5,4) [fill=black] {};
		\node[vertex] (p2) at (2.5,2) [fill=black] {};
		
		\node[vertex] (a1) at (0,3) [fill=black] {};
		\node[vertex] (a2) at (1,3) [fill=black] {};
		\node[vertex] (a3) at (2,3) [fill=black] {};
		\node[vertex] (a4) at (3,3) [fill=black] {};
		\node[vertex] (a5) at (4,3) [fill=black] {};
		\node[vertex,dashed] (a6) at (5,3) [fill=gray] {};
		
		\draw[thickedge] (p1)--(a1);
		\draw[thickedge] (p1)--(a2);
		\draw[thickedge] (p1)--(a3);
		\draw[thickedge] (p1)--(a4);
		\draw[thickedge] (p1)--(a5);
		\draw[thickedge,dashed] (p1)--(a6);
		
		\draw[thickedge] (p2)--(a1);
		\draw[thickedge] (p2)--(a2);
		\draw[thickedge] (p2)--(a3);
		\draw[thickedge] (p2)--(a4);
		\draw[thickedge] (p2)--(a5);
		\draw[thickedge,dashed] (p2)--(a6);

		\end{tikzpicture}
		\caption{The graph $ K_{2,n-2} $ which minimizes the Wiener index among all quadrangulations of order $ n $.} 
		\label{fig:qmin}

\end{figure}
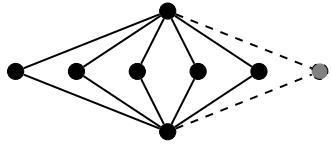

\begin{theorem} Assume that $n\geq 8$, $n\not=9$. The minimum Wiener index of $3$-connected quadrangulations of order $n$ is
\begin{equation*} 
4\left\lceil\frac{n}{2}\right\rceil^2+\left(\left\lceil\frac{n}{2}\right\rceil+21\right)\left(\left\lfloor\frac{n}{2}\right\rfloor-21\right) -5n +449=   \begin{cases}
    \frac{5n^2}{4} -5n + 8, & \text{if $n$ is even}, \\
\frac{5n^2}{4} -3n - \frac{49}{4}, & \text{if $n$ is odd}. \\
\end{cases}
\end{equation*}
The unique minimizer of the Wiener index among  $3$-connected quadrangulations  of order $n$ is $Q_n^3$, defined  in {\upshape Figure ~\ref{fig:evenq3cminwi}}.
\end{theorem}
The proof of this theorem is in Section~\ref{sunfl}, combining Lemma~\ref{Qn3lista}~\ref{fformula} and
Theorems~\ref{q3cminwieven}
and \ref{q3cminwiodd}. 
No $3$-connected quadrangulation exists of order $n\leq 7$ by Euler's formula,  and of order
$n=9$     by  Lemma~\ref{la:maxdegreeq3c}~\ref{case:c}.

\begin{theorem} 
Assume that $n\ge 12$, $n\ne 13$. The minimum Wiener index of $5$-connected triangulations of order $n$ is
 \begin{equation*} 
  2n\left\lceil\frac{n}{2}\right\rceil+\left(\left\lceil\frac{n}{2}\right\rceil+14\right)\left(\left\lfloor\frac{n}{2}\right\rfloor-14\right)-7n+208
  =  \begin{cases}
    \frac{5n^2}{4} -7n + 12, & \text{if $n$ is even}, \\
    \frac{5n^2}{4} -6n - \frac{9}{4}, & \text{if $n$ is odd}. \\
\end{cases}
 \end{equation*}
The unique minimizer of the Wiener index among  $5$-connected triangulations  of order $n\not=19$ is $T_n^5$, defined  on {\upshape Figure ~\ref{fig:t5cminwi}}, 
while for $n=19$, exactly two minimizers exist, namely $T_{19}^5$, and the  $5$-connected triangulation $X$ of order $19$, defined on
{\upshape  Figure ~\ref{fig:t5cotherodd1}}.
\end{theorem}
The proof of this theorem is in Section~\ref{mosaicsec}, 
combining Lemma~\ref{Tn5lista}~\ref{eformula},
Lemma~\ref{Qn3lista}~\ref{eformula} and Theorems \ref{t5cminwieven}
and \ref{t5cminwiodd}. 
No $5$-connected triangulation exists of order  $n\leq 11$ by Euler's formula, 
and of order $n=13$  by Lemma~\ref{la:mosaic}.

\section{Minimum Wiener Index of 3-Connected Quadrangulations} \label{sunfl}

Note that Euler's formula implies that there are no $3$-connected quadrangulations on fewer than $8$ vertices.

First, we define an auxiliary drawn graph, which we will use extensively in  this section. Let $v$ be a vertex of a $3$-connected quadrangulation $G$.
We define the {\it sunflower graph}  $S_v$ around $v$ (in the planar drawing of $G$),  as $v$ connected to its neighbors $u_1,\ldots,u_d$ (listed in the cyclic order of the drawing,
$d=d(v)$), and different vertices $w_1,\ldots,w_d$
where $w_i$ is connected to $u_i$ and $u_{i+1}$  (indices taken modulo $d$, see 
Figure ~\ref{fig:sunflowergraphq}). We understand $S_v$ as a part of the drawing of $G$.

First we need to show that such a graph, with distinct vertices,  exists in the drawing. We will also need some special properties of the sunflower graph, which will be shown in Lemma \ref{la:sunflower} below.

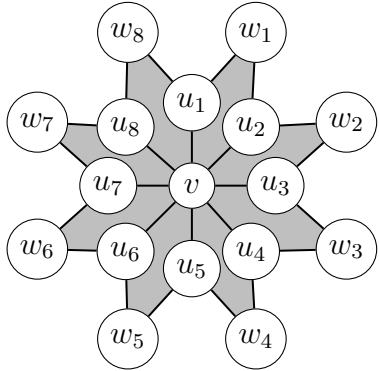
\begin{figure}[htbp]
		\centering
			\begin{tikzpicture}
			[scale=0.55,inner sep=1mm,
			vertex/.style={circle,draw,minimum size=6mm,fill=white},
			thickedge/.style={line width=0.75pt}] 
			\fill[gray!50] (0,2)--(1.53073, 3.69552)--(1.4142,1.4142)--(3.69552, 1.53073)--(2,0)--	(3.69552, -1.53073)--(1.4142,-1.5858)--(1.53073, -3.69552)--(0,-2)-- (-1.53073, -3.69552)--(-1.5858,-1.5858)--(-3.69552, -1.53073)--(-2,0)--(-3.69552, 1.53073)--(-1.5858,1.4142)--(-1.53073, 3.69552)--cycle;					
			\node[vertex] (v) at (0,0) {$v$};
			\node[vertex] (u1) at (0,2) {$u_{1}$};
			\node[vertex] (u2) at (1.4142,1.4142) {$u_{2}$};
			\node[vertex] (u3) at (2,0) {$u_{3}$};
			\node[vertex] (u4) at (1.4142,-1.5858) {$u_{4}$};
			\node[vertex] (u5) at (0,-2) {$u_{5}$};
			\node[vertex] (u6) at (-1.5858,-1.5858) {$u_{6}$};
			\node[vertex] (u7) at (-2,0) {$u_{7}$};
			\node[vertex] (u8) at (-1.5858,1.4142) {$u_{8}$};
			
			\node[vertex] (w1) at (1.53073, 3.69552) {$w_{1}$};
			\node[vertex] (w2) at (3.69552, 1.53073) {$w_{2}$};
			\node[vertex] (w3) at (3.69552, -1.53073) {$w_{3}$};
			\node[vertex] (w4) at (1.53073, -3.69552) {$w_{4}$};
			\node[vertex] (w5) at (-1.53073, -3.69552) {$w_{5}$};
			\node[vertex] (w6) at (-3.69552, -1.53073) {$w_{6}$};
			\node[vertex] (w7) at (-3.69552, 1.53073) {$w_{7}$};
			\node[vertex] (w8) at (-1.53073, 3.69552) {$w_{8}$};
            
        	\draw[thickedge] (v)--(u1);
            \draw[thickedge] (v)--(u2);
            \draw[thickedge] (v)--(u3);
            \draw[thickedge] (v)--(u4);
            \draw[thickedge] (v)--(u5);
            \draw[thickedge] (v)--(u6);
            \draw[thickedge] (v)--(u7);
            \draw[thickedge] (v)--(u8);
            
            \draw[thickedge] (u1)--(w1)--(u2);
            \draw[thickedge] (u2)--(w2)--(u3);
            \draw[thickedge] (u3)--(w3)--(u4);
            \draw[thickedge] (u4)--(w4)--(u5);
            \draw[thickedge] (u5)--(w5)--(u6);
            \draw[thickedge] (u6)--(w6)--(u7);
            \draw[thickedge] (u7)--(w7)--(u8);
            \draw[thickedge] (u8)--(w8)--(u1);

			\end{tikzpicture}
	    \caption{The sunflower graph $S_v$ around $v$ with $d(v)=8$. The region $\mathcal{R}_v$ is shaded.}
	\label{fig:sunflowergraphq}
\end{figure}
 
\begin{lemma}\label{la:sunflower} Assume that $Q$ is a drawing of a $3$-connected quadrangulation. Then, for any vertex $v$, $Q$ contains a sunflower graph $S_v$ with $2d(v)+1$ distinct vertices. Furthermore,
 the region ${\mathcal{R}}_v$ that contains $v$ and is bounded by the cycle $C_v=u_1w_1\ldots u_{d(v)}w_{d(v)}$ contains no vertices or edges that are not in $S_v$.
\end{lemma}

\begin{proof}
We know $\delta(Q)\geq 3$ by the $3$-connectedness. 
Label the neighbors of $v$ by $u_1, \dots, u_d$, in their planar cyclic order around $v$. For each pair of successive neighbors $u_i$ and $u_{i+1}$ (indices taken modulo $d$), let $w_i \not=v $ be their
common neighbor that completes the face $f_i$ that has $u_i,v,u_{i+1}$ on its boundary. This means, in particular, that the interior of $f_i$ has no vertices or edges. If $y$ is a neighbor of $u_i$ and $y\notin\{v,w_{i-1},w_i\}$ then $y$ must lie between $w_{i-1}$ and $w_i$ in the planar cyclic order around $u_i$, in particular, $w_{i-1}\ne w_i$
as $d(u_i)\ge 3$. As $Q$ is bipartite, $u_i\not= w_j$ for all $1\le i,j\le d$ .
We will show that each of the $w_i$'s must be distinct. As ${\mathcal{R}}_v$ is the union of the faces $f_i$, this finishes the proof.
Assume that  $w_i = w_j$ for some $j\ne i$. We already know that $j\notin\{i-1,i+1\}$ and the vertices $u_i,u_{i+1},u_j,u_{j+1}$ are all different.
We consider two regions of the planar drawing of $Q$: ${\mathcal{R}}_1$ is  bounded by the $4$-cycle $u_{i+1} v u_j w_i$ 
and does not contain the vertex $u_i$, and ${\mathcal{R}}_2$ is bounded by the $4$-cycle $u_ivu_{j+1}w_i$ and does not contain the vertex $u_{i+1}$. 
Thus the faces bounded by $u_ivu_{i+1}w_i$ and $u_jvu_{j+1}w_j$ are disjoint from ${\mathcal{R}}_1$ and ${\mathcal{R}}_2$.
The neighbors of $u_i$ that differ from $v$ and $w_{i}$ must lie in ${\mathcal{R}}_2$  and
the neighbors of $u_j$ that differ from $v$ and $w_{i}$ must lie in ${\mathcal{R}}_1$. 
Hence $\{v,w_{i}\}$ separates $u_i$ from $u_{j}$ (See Figure ~\ref{fig:cutsetq1}), contradicting
the fact that $Q$ is $3$-connected.
	\begin{figure}[htbp]
		\centering
			\begin{tikzpicture}
			[scale=0.6,xscale=0.75,inner sep=.5mm,
			vertex/.style={circle,draw,minimum size=7mm, fill=white},
			thickedge/.style={line width=0.75pt}] 
			\coordinate (V) at (5,1);
			\coordinate (W) at (5,5);
			\coordinate (U1) at (-1,3);
			\coordinate (U2) at (2,3);
			\coordinate (U3) at (8,3);
			\coordinate (U4) at (11,3);
			\fill[gray!59] (V)--(U1)--(W)--(U2)--cycle;
			\fill[gray!50] (V)--(U3)--(W)--(U4)--cycle;						
			\node at  (-5,2) {${\mathcal{R}}_2$};
			\node[vertex] (v) at (V) {$v$};
			\node[vertex] (ui) at (U1) {$u_{i}$};
			\node[vertex] (ui1) at (U2) {$u_{i+1}$};
			\node[vertex] (uj) at (U3) {$u_{j}$};
			\node[vertex] (uj1) at (U4) {$u_{j+1}$};
			\node[vertex] (w) at (W) {$w_i=w_j$};
			\node at (5,3) {${\mathcal{R}}_1$};
			\draw[thickedge] (v)--(ui);
			\draw[thickedge] (v)--(ui1);
			\draw[thickedge] (v)--(uj);
			\draw[thickedge] (v)--(uj1);
			\draw[thickedge] (w)--(ui);
			\draw[thickedge] (w)--(ui1);
			\draw[thickedge] (w)--(uj);
			\draw[thickedge] (w)--(uj1);
		\end{tikzpicture}
		\caption{2-element cutset appears in $S_v$, when $w_i=w_j$. The two faces $vu_iw_iu_{i+1}$ and $vu_jw_ju_{j+1}$ are shaded.}
		\label{fig:cutsetq1}
	\end{figure}
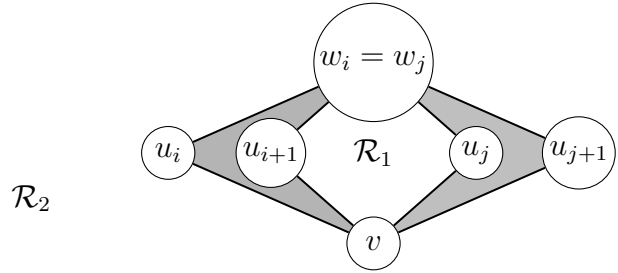
\end{proof}

\begin{lemma} \label{la:maxdegreeq3c}
    Assume that $G$ is a $3$-connected quadrangulation  with partite sets $A,B$.  Then
    \begin{enumerate}[label={\upshape (\alph*)}]
    \item \label{case:a} $
    \Delta(G) \leq \min\{|A|-1,|B|-1\}$;
       \item  \label{case:b}
    If $|B|<|A|$, then for all $x\in A$ we have $d(x) \leq |B|-2$;  and
    \item  \label{case:c} $|V(G)|\not= 9$, i.e.,   no $3$-connected quadrangulations exist on $9$ vertices.
    \end{enumerate}
\end{lemma}

\begin{proof} \ref{case:a} Let $v$ be a vertex with degree $\Delta=\Delta(G)$, we may assume   $v\in B$. As the sunflower  $S_v$ is a subgraph of the planar drawing of $G$, 
 $\Delta \leq \min(|B|-1,|A|)$. We are done unless $|B|>|A|=\Delta$, so assume that is the case.
As $A=N(v)$, all neighbors of the vertices of $B$ lie in $N(v)$, in particular, every $w_i$ has at least $3$ neighbors in $A$.
For each $i$ let $k_i$ be the largest positive integer such that
$w_i$ has no neighbors in the set $\{u_{i-t}: 1\le t\le k_i-1\}\cup \{u_{i+1+t}: 1\le t\le k_i-1\}$.
Since for $k=1$ the sets $\{u_{i-t}: 1\le t\le k-1\}$ and $\{u_{i+1+t}: 1\le t\le k-1\}$ are empty, such positive integers exist, they have an upper bound from
the fact that $w_i$ has at least $3$ neighbors in $A$, and
for the largest such integer $k_i$ we have that at least one of $u_{i-k_i},u_{i+1+k_i}$ is a neighbor of $w_i$ that is different from $u_i,u_{i+1}$.
Choose $i_0$ such that $k=k_{i_0}$ is minimal amongst the $k_i$. By renumbering the $u_i$ if necessary and changing the direction of the cyclic order
we can assume that $i_0=1$ and $w_1$ is connected to $u_1,u_2,u_{2+k}$ but none of $u_{1-t},u_{2+t}$ for all $1\le t\le k-1$. Let ${\mathcal{R}}$ 
be the region of the sphere bounded by the $4$-cycle $u_2vu_{2+k}w_1$ that does not contain $u_1$. Consider $w_2$.
By the definition of $w_2$ and the minimality of $k$, $w_2$ lies in ${\mathcal{R}}$ and  it has at least one neighbor $u_j$ that does not lie in ${\mathcal{R}}$. The edge $w_2u_j$ must cross
the boundary of ${\mathcal{R}}$, which contradicts the planarity of
$G$. Thus, we have $\Delta(G) \leq \min\{|A|-1,|B|-1\}$, as claimed.

To prove the case \ref{case:b}, assume $|B|<|A|$, i.e. $n>2|B|$. We already know that $\Delta(G)\le |B|-1$. Assume that $A$ contains a vertex $v$ of degree $|B|-1$.
Since $|A|=n-|B|$ and all other vertices of $A$ have degree at least $3$, we have that $2n-4\ge (|B|-1)+3(n-|B|-1)=3n-2|B|-4$, 
so $n\le 2|B|$, a contradiction.

To prove the case \ref{case:c}, assume to the contrary that $G$ has $9$ vertices and partite sets $A,B$. We may assume $|B|<|A|$, and therefore $|B|\le 4$. Then every vertex in $A$ has degree at most $2$, a contradiction.
\end{proof}

\begin{lemma} \label{la:dist2q}
In a $3$-connected quadrangulation $G$ of order $n$, 
the number of unordered pairs of vertices at distance  2 is at most $$\frac{1}{2}\sum_v d^2(v)-4(n-2).$$
This estimate is exact precisely when $G$ has no non-facial $4$-cycles.
\end{lemma}

\begin{proof}  Euler's Formula gives us that any quadrangulation on $n$ vertices has $2(n-2)$ edges and $n-2$ faces.
The number of 2-paths in $G$ is equal to $\sum_v \binom{d(v)}{2}=\frac{1}{2}\sum_v d^2(v) -2(n-2)$. This sum, however, overcounts the number of pairs of vertices distance 2 apart. In a 
$3$-connected quadrangulation, two faces cannot share two consecutive edges from their boundaries. Thus, for each face, we are double counting the two pairs of vertices distance 2 apart, and so we may safely subtract $2(n-2)$. There are pairs of vertices which we have double counted even after the substraction precisely when there are non-facial $4$-cycles.
\end{proof}

	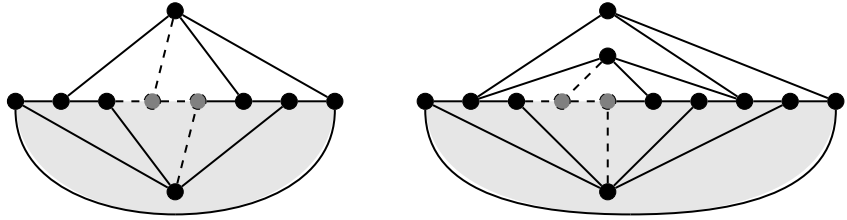
\begin{figure}[ht]
		\centering
			\begin{tikzpicture}
			[scale=0.6,inner sep=0.75mm, 
			vertex/.style={circle,draw}, 
			thickedge/.style={line width=0.75pt}] 
			
			
			\fill[gray!20] (7,3)--(0,3) to [out=-90,in=180] (3.5,.5) to [out=0,in=-90] (7,3);

			\node[vertex] (p1) at (3.5,5) [fill=black] {};
			\node[vertex] (p2) at (3.5,1) [fill=black] {};

			\node[vertex] (a1) at (0,3) [fill=black] {};
			\node[vertex] (a2) at (1,3) [fill=black] {};
			\node[vertex] (a3) at (2,3) [fill=black] {};
			\node[vertex,dashed] (a4) at (3,3) [fill=gray] {};
			\node[vertex,dashed] (a5) at (4,3) [fill=gray] {};
			\node[vertex] (a6) at (5,3) [fill=black] {};
			\node[vertex] (a7) at (6,3) [fill=black] {};
			\node[vertex] (a8) at (7,3) [fill=black] {};
			
			\draw[thickedge] (p1)--(a2);
			\draw[thickedge,dashed] (p1)--(a4);
			\draw[thickedge] (a1)--(a2)--(a3);
			\draw[thickedge,dashed] (a3)--(a4)--(a5);
			\draw[thickedge] (a5)--(a6)--(a7)--(a8);
			\draw[thickedge] (p1)--(a6);
			\draw[thickedge] (p1)--(a8);
			
			\draw[thickedge] (a1) to[out=-90,in=180] (3.5,.5);
			\draw[thickedge] (a8) to[out=-90,in=0] (3.5,.5);
			
			\draw[thickedge] (p2)--(a1);
			\draw[thickedge] (p2)--(a3);
			\draw[thickedge,dashed] (p2)--(a5);
			\draw[thickedge] (p2)--(a7);
			
			\end{tikzpicture}
			\qquad
			\begin{tikzpicture}
			[scale=0.6,inner sep=0.75mm, 
			vertex/.style={circle,draw}, 
			thickedge/.style={line width=0.75pt}] 
			
			
			
			\fill[gray!20] (9,3)--(0,3) to[out=-90,in=180] (4.5,.5) to [out=0,in=-90] (9,3);

			\node[vertex] (p1) at (4,4) [fill=black] {};
			\node[vertex] (p2) at (4,1) [fill=black] {};
	
			\node[vertex] (a1) at (0,3) [fill=black] {};
			\node[vertex] (a2) at (1,3) [fill=black] {};
			\node[vertex] (a3) at (2,3) [fill=black] {};
			\node[vertex,dashed] (a4) at (3,3) [fill=gray] {};
			\node[vertex,dashed] (a5) at (4,3) [fill=gray] {};
			\node[vertex] (a6) at (5,3) [fill=black] {};
			\node[vertex] (a7) at (6,3) [fill=black] {};
			\node[vertex] (a8) at (7,3) [fill=black] {};
			\node[vertex] (a9) at (8,3) [fill=black] {};
			\node[vertex] (a10) at (9,3) [fill=black] {};
			\node[vertex] (q1) at (4,5) [fill=black] {};

			\draw[thickedge] (p1)--(a2);
			\draw[thickedge,dashed] (p1)--(a4);
			\draw[thickedge] (a1)--(a2)--(a3);
			\draw[thickedge,dashed] (a3)--(a4)--(a5);
			\draw[thickedge] (a5)--(a6)--(a7)--(a8)--(a9)--(a10);
			\draw[thickedge] (p1)--(a6);
			\draw[thickedge] (p1)--(a8);

			\draw[thickedge] (a1) to[out=-90,in=180] (4.5,.5);
			\draw[thickedge] (a10) to[out=-90,in=0] (4.5,.5);

			\draw[thickedge] (a2)--(q1);
			\draw[thickedge] (a8)--(q1);
			\draw[thickedge] (a10)--(q1);		
			
			\draw[thickedge] (p2)--(a1);
			\draw[thickedge] (p2)--(a3);
			\draw[thickedge,dashed] (p2)--(a5);
			\draw[thickedge] (p2)--(a7);
			\draw[thickedge] (p2)--(a9);
			
			\end{tikzpicture}
			\caption{The quadrangulation $ Q_n^3 $ of order
			 $ n = 2k \geq 8$ (left) and $n=2k+1\geq 11$  (right),   which minimizes the Wiener index
			 among all $3$-connected quadrangulations of order $n$. Gray vertices and dashed edges indicate the pattern to be repeated. The light gray regions are the sunflower graphs around a maximum degree vertex. } 
			\label{fig:evenq3cminwi}
	\end{figure}

\begin{lemma} \label{la:optim} Let $Q$ be a $3$-connected quadrangulation of order $n$, with partite sets $A,B$. Then
\begin{equation} 
W(Q)  \geq  2n^2 +2n -8 - |A||B|- \sum_v d^2(v).  \label{tooptimize}
\end{equation}
Equality holds in {\upshape (\ref{tooptimize})} precisely when the diameter of $Q$ is at most $4$ and $Q$ has no non-facial $4$-cycles. 
\end{lemma}
\begin{proof}
Let $Q$ be an arbitrary $3$-connected quadrangulation on $n$ vertices, with partite sets $A,B$. Let $D_i$ denote the number of unordered pairs of vertices at distance $i$ in $Q$. Clearly $W(Q)=\sum_i i\cdot D_i$. Observe that $D_1=2n-4$, the number of edges; $D_2\leq \frac{1}{2}\sum_v d^2(v)-4(n-2)$ by Lemma~\ref{la:dist2q};
$D_2+ D_4+D_6+D_8\cdots = \binom{|A|}{2}+ \binom{|B|}{2}$, as pairs of vertices are at even distance precisely when they are from the same  partite set; and finally, $D_1+ D_3+D_5+D_7+\cdots = |A|\cdot |B|$, as pairs of vertices are at odd distance
precisely when they are from different partite sets.

Combining all this information with the identity $|A|+|B|=n$, we obtain that
\begin{eqnarray*} 
W(Q) & \geq & (2n-4)+2 D_2+3 \biggl[ |A|\cdot |B|-(2n-4)\biggl]+4\biggl[ \binom{|A|}{2}+ \binom{ |B|}{2}-D_2\biggl]\\
&=& 2n^2 -2n - |A||B|-2(D_2+2n-4)\\
&\geq& 2n^2 +2n -8 - |A||B|- \sum_v d^2(v).  
\end{eqnarray*}
The first inequality  in the displayed formula is an equality precisely when the diameter of $Q$ is at most $4$, and the second inequality is an equality precisely when $Q$ has no non-facial $4$-cycles.
\end{proof}
The $3$-connected quadrangulation $Q_n^3$ of order $n\geq 8$, $n\not=9$  is defined  in Figure ~\ref{fig:evenq3cminwi}. The following lemma is easy to verify, and we leave 
the details to the reader.
 \begin{lemma} \label{Qn3lista}
Assume that $n\geq 8$, $n\not=9$.
\begin{enumerate}[label={\upshape (\alph*)}]
\item\label{qstart} $Q_n^3$ is a $3$-connected quadrangulation.
\item  $Q_n^3$  has no non-facial 4-cycle.
\item \label{qevendegseq} If $n$ is even, $Q_n^3$ has diameter $3$ and degree sequence $\frac{n}{2}-1, \frac{n}{2}-1, 3,\ldots,3$.
\item \label{qodddegseq} If $n$ is odd, $Q_n^3$ has diameter $4$ and degree sequence $\lfloor\frac{n}{2}\rfloor-1,\lfloor\frac{n}{2}\rfloor-2,4,4,3,\ldots, 3$. (For $n=11$, the terms in this sequence are {\sl not} in decreasing order.)
\item  \label{fformula} \begin{equation*} 
    W(Q_n^{3}) =\begin{cases}
    \frac{5n^2}{4} -5n + 8, & \text{if $n$ is even}, \\
\frac{5n^2}{4} -3n - \frac{49}{4}, & \text{if $n$ is odd}. \\
\end{cases}
 \end{equation*}
\end{enumerate}
\end{lemma}

The following is obvious, and we make use of it frequently
\begin{lemma}\label{la:proc}Assume that $\sum_{i=1}^n x_i=a>0$ is given, where the $x_i$'s are required to be integers from the interval $[b,c]$ with $0\le b$, and we have to maximize $\sum_{i=1}^n x_i^2$.  
As long as for some $(i\ne j$) we have $b+1\le x_j\le x_i\le c-1$, we can increase the sum of squares while keeping the conditions by changing $x_i$ to $x_i+1$ and $x_j$ to $x_j-1$.
\end{lemma}

\begin{theorem} \label{q3cminwieven} Assume that the number $n\geq 8$ is even.
    The quadrangulation $ Q_n^3 $ defined in {\upshape Figure ~\ref{fig:evenq3cminwi}} minimizes the Wiener index among all $3$-connected quadrangulations of order $n$.
     Moreover, up to isomorphism, this minimizer is unique. 
\end{theorem}

\begin{proof} Let $Q$ be an arbitrary $3$-connected quadrangulation on $n=2k$ vertices, with partite sets $A,B$.
Since $Q$ is $3$-connected, for all $v$, we have $ d(v)\geq 3$, and by Lemma~\ref{la:maxdegreeq3c}, $d(v)\le \Delta(Q)\le\min(|A|-1,|B|-1)\le \frac{n}{2}-1$. 
By Lemma~\ref{la:proc}  and Lemma~\ref{Qn3lista}~\ref{qevendegseq}, $\sum_{v\in V(Q)}d^2(v)\le\sum_{v\in V(Q_n^3)}d^2(v)$ with equality precisely when $Q$ has the same degree sequence as
$Q_n^3$. 
Also, $|A|\cdot|B|\le\frac{n^2}{4}$ with equality precicely when
$|A|=|B|=\frac{n}{2}$.
Lemma~\ref{la:optim} gives that $W(Q)\ge W(Q_n^3)$ with equality precisely when $Q$ has the same degree sequence as
$Q_n^3$, $|A|=|B|=\frac{n}{2}$, $Q$ has diameter at most $4$ and no nonfacial $4$-cycles. In particular, $Q_n^3$ minimizes the Wiener index among $n$-vertex $3$-connected quadrangulations.

We will show that the extremal quadrangulation is in fact unique. Assume that $W(Q)=W(Q_n^3)$, so $Q$ has the same degree sequence as $Q_n^3$ and $|A|=|B|=k$ vertices. Then in both $A$ and $B$ we have 
$k-1$ vertices of degree $3$, and the remaining  one vertex must have degree $k-1$  ($k-1\ge 3$). 

As before, let $v$ be a vertex of maximum degree $k-1$, and construct the sunflower graph $S_v$ around $v$. Since $S_v$ has exactly $n-1$ vertices, $Q$ has one additional vertex $v'$. 
This vertex $v'$ is in the same partite class as the $u_i$ vertices, and differs from $v$. Each of the $w_i$ has one edge not in $S_v$ incident upon it, connecting them to either $v'$ or one of the $u_j$.
If all vertices $u_i$ have degree $3$, then $v'$ has degree $k-1\ge 3$, and it is adjacent to all $w_i$ (in which case we have $Q_n^3$). Otherwise the degree of $v'$ is $3$ and exactly one of the $u_i$ 
(say $u_2$) has degree $k-1>3$ in $Q$.
Assume that the latter is the case.  
As $w_1$ and $w_2$ have an edge not in $E(S_v)$ incident upon them, and $w_1u_2,w_2u_2\in E(S_v)$, both $w_1$ and $w_2$ are adjacent to $v'$. 
Thus, $v'w_1u_2w_2$ bounds a facial region $\mathcal{R}$. As $w_1u_2w_2$, of which $u_2$ is an internal vertex, is the common boundary of $\mathcal{R}_v$ and $\mathcal{R}$, 
$u_2$ cannot have any edge outside of $S_v$ incident upon it, a contradiction.
\end{proof}

\begin{lemma} \label{la:q3cwiodd} Assume $n=2k+1$, 
and let $Q$ be a $3$-connected quadrangulation of order $n$,
with partite sets $A,B$. If
\begin{equation}
\sum_{v\in V(Q)}d^2(v)< 2k^2+12k+10, \label{keymin} 
\end{equation}
then $W(Q)>W(Q_n^3)$. 
If $\Delta(Q)\le k-2$, then $W(Q)>W(Q_n^3)$.
\end{lemma}

\begin{proof}
First note that by Lemma~\ref{Qn3lista}~\ref{qodddegseq}
\begin{eqnarray*}
\sum_{x\in V(Q_n^3)}d^2(x)&=&(k-1)^2+(k-2)^2+2\cdot 4^2+3^2(2k-3)
=2k^2+12k+10,
\end{eqnarray*}
and also $|A||B|\le k(k+1)$ (note that $k,k+1$ are the sizes of the partite classes in $Q_n^3$), so if (\ref{keymin}) holds, then by Lemma~\ref{la:optim} and Lemma~\ref{Qn3lista}~\ref{qodddegseq} we have $W(Q)>W(Q_n)$.

Since $Q$ has odd number of vertices and minimum degree $3$, the Handshaking Lemma implies $\Delta(Q)\ge 4$.
Assume now that $\Delta(Q)\le k-2$, so $k\ge 6$.
Let $x_1,\ldots,x_{2k+1}$ be a sequence of integers that maximizes 
$\sum x_i^2$ subject to the conditions that that $\sum x_i=4(n-2)$ and $3\le x_i\le k-2$ .
If $k=6$, the maximizing sequence is $4,4,4,4,4,3,\ldots,3$ of length 13, and if $k=7$ the maximizing sequence by Lemma~\ref{la:proc} is $5,5,5,4,3,\dots 3$ of length 15. In both of these cases, we have  
$\sum x_i^2<2k^2+12k+10$.
For $k\ge 8$, Lemma~\ref{la:proc} gives that the maximizing  sequence is $k-2,k-2,6,3,3,\ldots,3$, so 
$$\sum_{x\in V(Q)}d^2(x)\le \sum_{i=1}^{2k+1}x_i^{2}=2(k-2)^2+6^2+3^2(2k-2)=2k^2+10k+26\le 2k^2+12k+10.$$
Therefore $W(Q)> W(Q_n^3)$ unless $k=8$ and the degree sequence of $Q$ is $6,6,6,3,3,\ldots,3$.

So for the rest of this proof $k=8$, the degree sequence of $Q$ is $6,6,6,3,3,\ldots,3$ and is of length $17$.
By Lemma~\ref{la:optim} if $W(Q)\le W(Q_n^3)$, then $W(Q)=W(Q_{17}^3)$,
the diameter of $Q$ is $4$, $Q$ has no nonfacial $4$-cycles, $|A|=9$ and $|B|=8$. We will show that such a $Q$ does not exist, which finishes the proof.

Because the sum of the degrees of the vertices in each partite class must be the same (in this case, $30$), 
$B$ contains exactly two of the degree $6$ vertices. 
Let $v\in B$ with degree $6$, consider the sunflower $S_v$, and label the $4$ vertices outside $S_v$ by $x,y_1,y_2,y_3$ such that
$B=\{v,w_1,\ldots,w_6,x\}$ and $A=\{u_1,\ldots,u_6,y_1,y_2,y_3\}$. Since $d(x)\in\{3,6\}$, $N(x)\subseteq A$ and at most one of the $u_i$ has
an edge not from $S_v$ incident upon it, we have $d(x)=3$ and without loss of generality $y_1,y_2\in N(x)$.

Assume first that $N(x)=\{y_1,y_2,y_3\}$ and consider the sunflower $S_x$.  Let $j_i$ be chosen such that $w_{j_i}$ is the common neighbor of $y_i$ and $y_{i+1}$ (indices taken modulo $3$) in
$S_x$. Then each of the $w_{j_i}$ are different and have degree at least $4$ in $Q$, a contradiction.

So we can assume without loss of generality that $N(x)=\{u_1,y_1,y_2\}$. 
Then the unique degree $6$ vertex in $A$ is $u_1$, so there are two different indices $t_1$ and $t_2$ such that $w_{t_i}\in N(u_1)\setminus\{w_1,w_6\}$. 
For $i\in\{1,2\}$ let $z_i$ be the common neighbor of $u_1$ and $y_i$ in the sunflower $S_x$, and let $z_3$ be the common neighbor of $y_1$ and $y_2$ in $S_x$. 
Then $\{z_1,z_2\}\subseteq \{w_1,w_2,w_{t_1},w_{t_2}\}$ and $z_3\in \{w_1,\ldots,w_6\}\setminus\{z_1,z_2\}$. In particular, the degree of $z_3$ in $Q$ is at least $4$, therefore $z_1$ and $z_2$ must have degree
$3$ in $Q$.
If
$w_{t_i}\in\{z_1,z_2\}$, then $w_{t_i}$ has degree at least $4$ and consequently degree $6$. This gives  $\{z_1,z_2\}\cap\{w_{t_1},w_{t_2}\}=\emptyset$. 
Therefore without loss of generality $w_1=z_1$, $w_6=z_2$ and the $u_1w_{t_i}$ edges
cannot run inside $\mathcal{R}_v$ or any of the faces bounded the $4$-cycles $u_1w_1y_1xu_1$ and
$u_1w_6y_2xu_1$, which leaves them no place to be, a contradiction.
\end{proof}

\begin{theorem} \label{q3cminwiodd} Assume that the number $n\geq 11$ is odd.
    The quadrangulation $ Q_n^3 $ in {\upshape Figure ~\ref{fig:evenq3cminwi}} minimizes the Wiener index among all $3$-connected quadrangulations of order $n$.
     Moreover, up to isomorphism, this minimizer is unique. 
\end{theorem}
\noindent 

\begin{proof}
Let $n=2k+1$ and assume that $Q$ is a $3$-connected quadrangulation on $n$ vertices of minimum Wiener index, and with partite sets $A,B$, such that $|A|>|B|$.

First we want to show that $|A|=k+1$, $|B|=k$, and the degree sequence of $Q$ restricted to the partite sets is the same as the degree sequence of $Q_n^3$ restricted to its partite sets.

We have $|A|\geq k+1$ and $|B|\leq k$.
Lemma~\ref{la:maxdegreeq3c}~\ref{case:a} gives $\Delta(Q)\leq |B|-1\leq k-1$. Lemma~\ref{la:q3cwiodd}  gives $\Delta(Q)=k-1$, 
which in turn shows $|B|=k$ and $|A|=k+1$. In addition, if $d(v)=k-1$, Lemma~\ref{la:maxdegreeq3c}~\ref{case:b} gives $v\in B$.

As the degree sequence of quadrangulations is unique for $n=11$ under the condition that every degree is $3$ or $4$, we may assume now that $n\ge 13$, i.e., $k\ge 6$.
As the sum of the $k$ degrees in $B$ is the number of edges $2n-4=4k-2$, and every  degree is at least 3, only two degree sequences are possible for $B$:
$(k-1,5,3,3,\ldots,3)$ or $(k-1,4,4,3,\ldots, 3)$. We claim that $\Delta(A)$, the maximum degree of a vertex in $A$ is $k-2$. Lemma~\ref{la:maxdegreeq3c}~\ref{case:b}
showed $\Delta(A)\leq |B|-2=k-2$. 

Assume for contradiction that $\Delta(A)\leq k-3$. Since the minimum degree is at least $3$ and  $\sum_{x\in A} d(x)=4k-2$, for $k=6$ we get
that $3\cdot 7=4\cdot 6-2$, a contradiction. Therefore we have that $k\ge 7$,
$$\sum_{x\in A}d^2(x)\leq (k-3)^2+4^2+3^2(k-1)=k^2+3k+16,$$
and
\begin{eqnarray*}
\sum_{x\in V(Q)}d^2(x)&\leq &k^2+3k+16+(k-1)^2+5^2+3^2(k-2)=2k^2+10k+24.\\
\end{eqnarray*}
By Lemma~\ref{la:q3cwiodd} $W(Q)>W(Q_n^3)$ when $k\ge 8$ so we may assume that $k=7$. In particular, for $k\ge 8$ the degree sequence of $A$ is $(k-2,3,\ldots,3)$.

If $k=7$, the degree sequence of $A$ is $(4,4,3,3,3,3,3,3)$ and the degree sequence of $B$ is $(6,4,4,3,3,3,3)$ then
$
\sum_{x\in V(Q)} d^2(x) =190<192=2\cdot7^2+12\cdot 7+10,
$
and Lemma~\ref{la:q3cwiodd}  contradicts the minimality of the Wiener index of $Q$. 

Hence the only case that remains to be checked is when $k=7$, the degree sequence of $A$ is $(4,4,3,3,3,3,3,3)$ and the degree sequence of $B$ is $(6,5,3,3,3,3,3)$. 
In this case 
$
\sum_{x\in V(Q)} d^2(x) =192=\sum_{x\in V(Q_{15}^3)} d^2(x)
$, so by Lemma~\ref{la:optim} and Lemma~\ref{Qn3lista}~\ref{qstart},~\ref{qodddegseq} the minimality of $W(Q)$ implies that $Q$ has no nonfacial $4$-cycles.
Let $v\in B$, $d(v)=6$ and consider the sunflower $S_v$. Let $x,y$ be the vertices outside $S_v$. Then $B=\{v,w_1,\ldots,w_{6}\}$ and
$A=\{u_1,\ldots,u_{6},x,y\}$, without loss of generality $d(w_1)=5$, and the rest of the $w_i$ have degree $3$. Therefore there is an $i\in\{3,4,5,6\}$ such that $w_1$ is adjacent to $u_i$. Since $w_1$ and $u_i$
cuts $C_v$ into two paths, one contains $w_2$ and the other $w_6$, the vertices $w_2$ and $w_6$ lie inside two different regions bounded by the $4$-cycle 
$w_1u_jvu_1w_1$. As this cycle is nonfacial, we have a contradiction.

	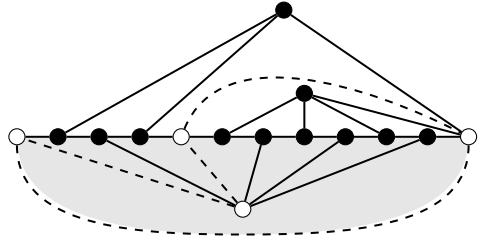
\begin{figure}[ht]
		\centering
			\begin{tikzpicture}
			[scale=0.6,xscale=.9,yscale=.8,inner sep=0.75mm, 
			vertex/.style={circle,draw}, 
			thickedge/.style={line width=0.75pt}] 

			\fill[gray!20] (0,3)--(11,3) to [out=-90, in=0] (5.5,.3) to [out=180,in=-90] (0,3);
			\node[vertex] (p1) at (7,4.2) [fill=black] {};
			\node[vertex] (p2) at (5.5,1) [fill=white] {};			
			\node[vertex] (z1) at (6.5,6.5) [fill=black] {};

			\node[vertex] (a1) at (0,3) [fill=white] {};
			\node[vertex] (a2) at (1,3) [fill=black] {};
			\node[vertex] (a3) at (2,3) [fill=black] {};
			\node[vertex] (a4) at (3,3) [fill=black] {};
			\node[vertex] (a5) at (4,3) [fill=white] {};
			\node[vertex] (a6) at (5,3) [fill=black] {};
			\node[vertex] (a7) at (6,3) [fill=black] {};
			\node[vertex] (a8) at (7,3) [fill=black] {};
			\node[vertex] (a9) at (8,3) [fill=black] {};
			\node[vertex] (a10) at (9,3) [fill=black] {};
			\node[vertex] (a11) at (10,3) [fill=black] {};
			\node[vertex] (a12) at (11,3) [fill=white] {};
			
			\draw[thickedge] (p1)--(a12);
			\draw[thickedge] (p1)--(a10);
			\draw[thickedge] (a1)--(a2)--(a3);
			\draw[thickedge] (a3)--(a4)--(a5);
			\draw[thickedge] (a5)--(a6)--(a7)--(a8)--(a9)--(a10)--(a11)--(a12);
			\draw[thickedge] (p1)--(a6);
			\draw[thickedge] (p1)--(a8);
						
			\draw[thickedge,dashed] (a1) to[out=-90,in=180] (5.5,.3);
			\draw[thickedge,dashed] (a12) to[out=-90,in=0] (5.5,.3);
			\draw[thickedge,dashed] (a5) to [out=70, in=150] (a12);
					
			\draw[thickedge] (z1)--(a2);

			\draw[thickedge,dashed] (p2)--(a1);
			\draw[thickedge] (p2)--(a3);
			\draw[thickedge,dashed] (p2)--(a5);
			\draw[thickedge] (p2)--(a7);
			\draw[thickedge] (p2)--(a9);
			\draw[thickedge] (p2)--(a11);

			\draw[thickedge] (z1)--(a4);
			\draw[thickedge] (z1)--(a12);

			\end{tikzpicture}
			\caption{The  quadrangulation $Y$  of order $13$ with Wiener index $164$. The gray region shows the sunflower around a maximal degree vertex. 
			The white vertices and the dotted edges  form one of the non-facial $4$-cycles.} 	
		\label{fig:YZ}
	\end{figure}

So we have that the degree sequence of $Q$ restricted to $A$ is the same as the degree sequence of $Q_n^3$ restricted to its $A$, i.e. $(k-2,3,\ldots,3)$.
We need to figure out what the degree sequence of $Q$ restricted to $B$ is. Assume $k\ge 6$, and $B$ has degree sequence $(k-1,5,3,3,\ldots,3)$.
Referring to the sunflower graph $S_v$ at vertex $v$, where  $d(v)=k-1$, we have 
$B=\{v,w_1,\ldots,w_{k-1}\}$ and $A=\{u_1,\ldots,u_{k-1},x,y\}$.
We can assume without loss of generality that $w_1\in B$ has degree 5, and for $i:2\le i\le k-1$ set $z_i$ be the unique vertex in $N(w_i)\setminus\{u_i,u_{i+1}\}$.
 Since $w_1$ is adjacent to $3$ vertices of $A\setminus\{u_1,u_2\}$, it is adjacent to at least one (and at most three) vertices in $\{u_i:3\le i\le k-1\}$,
 and consequently these vertices have degree
at least $4$. As $A$ has a single vertex with degree more than 3, we conclude that there is a unique $j:3\le j\le k-1$
that $w_1$ is adjacent to $u_j$, $d(u_j)=k-2$ and $d(x)=d(y)=3$, and $w_1$ is adjacent to $x$ and $y$. In addition, for $i:2\le i\le k-1$ we have that
$z_i\in\{x,y,u_j\}$; in particular $z_{j-1},z_j\in\{x,y\}$. We may assume without loss of generality that $z_{j-1}=x$.

Let $\mathcal{P}$ be the region bounded by $C_v$ that is different from $\mathcal{R}_v$, and let ${\mathcal{R}}_1$ and ${\mathcal{R}}_2$  be the two subregions that
the edge $w_1u_j$ cuts $\mathcal{P}$ into; without loss of generality the boundary of $\mathcal{R}_1$ is the cycle $w_1u_2w_2u_{3}\ldots u_j$.
Now ${\mathcal{R}}_1,{\mathcal{R}}_2$ and ${\mathcal{R}}_v$ share only vertices on the boundary, and the common boundary of ${\mathcal{R}}_1$ and ${\mathcal{R}}_2$ 
is the edge $u_jw_1$.
${\mathcal{R}}_1$ has $j-2\ge 1$ vertices $w_2,\ldots w_{j-1}$ from $B\setminus\{w_1\}$ on its common boundary with ${\mathcal{R}}_v$, and for
$i:2\le i\le j-1$ the vertex $z_i$ lies in $\mathcal{R}_1$ (inside or on the boundary). Since $z_{j-1}=x$, $x$ is inside $\mathcal{R}_1$. Let $\mathcal{Q}$ be the subregion of
$\mathcal{R}_1$ bounded by the cycle  $w_1u_2w_2\cdots u_{j-1}w_{j-1}xw_1$. Then for $i:2\le i\le j-2$ the vertex $z_i$ lies in $\mathcal{Q}$, so $z_i\in\{x,y\}$.
${\mathcal{R}}_2$ has $k-j\ge 1$ vertices $w_{j},w_{j+1},\ldots,w_{k-1}$ from $B\setminus\{w_1\}$ on its common boundary with 
${\mathcal{R}}_v$ and for $i:j\le i\le k-1$, $z_i$ lies in $\mathcal{R}_2$ (inside or on the boundary). Since $z_j\in\{x,y\}$ and $x$ is inside $\mathcal{R}_1$, this implies
that $z_j=y$, $y$ is inside $\mathcal{R}_2$, for all $i:j\le i\le k-1$ we have $z_i\in\{u_j,y\}$ and for all $i:2\le i\le j-1$ we have $z_i=x$. Similar logic as before gives that
for all $i:j\le i\le k-1$ we have $z_i=y$
 Since $d(x)=d(y)=3$, this means $3=j-1=k-j+1$, so $j=4$ and $k=6$.
We have $Q\simeq Y$ (see Figure~\ref{fig:YZ}) and $W(Q)=164> 160= W(Q_{13}^3)$, a contradiction.

For the rest of the proof we assume that $n\ge 11$, so $k\ge 5$.
The integer sequence that maximizes the sum of squares, and satisfies the conditions
we have for the degree sequence of $G$ in $A$ (respectively $B$) is
$k-2,3,\ldots,3$ (respectively $k-1,4,4,3,\ldots, 3$), the degree sequence of $Q_n^3$. Since $Q_n^3$ is a $3$-connected quadrangulation with diameter at most $4$, this shows
 that $W(Q_n^3)$ is minimal, and
the degree sequence of $Q$ is the same as the degree sequence of $Q_n^3$, and furthermore, the degree sequences of their respective 
partite sets are the same. Last, we need to show that
$Q\simeq Q_n^3$.

Let $v\in Q$ with $d(v)=k-1$, and consider the sunflower $S_v$ around $v$. Let $x,y$ be the vertices of $Q$ not in $S_v$.
Then $B=\{v,w_1,w_2,\ldots,w_{k-1}\}$,  $A=\{u_1,\ldots,u_{k-1},x,y\}$, and without loss of generality the two vertices of degree $4$ in  $B$ are $w_1$ and $w_j$.

 If every $u_i$ has degree $3$ (this must happen in particular when $k=5$ and vertices of $A$ all have degree $3$), then
none of the $u_i$ has a neighbor outside of $S_v$. In this case
$w_1$ and $w_j$ must both be adjacent to $x$ and $y$. 
Without loss of generality  the region ${\mathcal{R}}$  bounded by the cycle $xw_1u_2w_2\ldots u_jw_jx$ that does not contain $v$ contains $y$. (Otherwise we  exchange the name of
$x$ and $y$). 
If $j=k-1$, then $x$ is on the interior of the $4$-cycle $yw_{k-1}u_1w_1y$ that does not contain $v$, and the degree of $x$ can only be $2$, which is a contradiction. If $j=2$, then ${\mathcal{R}}$ is bounded by a $4$-cycle and $y$ can have only degree $2$, a contradiction.
So $3\le j\le k-2$, the  $k-1-j\ge 1$ vertices $w_{j+1},w_{j+2},\ldots,w_{k-1}$ must have $x$ as their third neighbor, and the 
 the $j-2\ge 1$ vertices $w_{2},w_{3},\ldots, w_{j-1}$ must have $y$ as their third neighbor.
 So $\{d(x),d(y)\}=\{k+1-j,j\}=\{3,k-2\}$, which gives $j\in\{3,k-2\}$. This is precisely the graph $Q_n^3$.
 
Now let $i$ be chosen so $u_i$ have degree greater than $3$. 
As all but one of the vertices of $A$ have degree $3$, for $j\ne i$ we have $d(u_j)=3$, $u_j$ has the same neighbors in $Q$ and $S_v$, $d(u_i)=k-2$, $d(x)=d(y)=3$, and
$k\ge 6$.
This means that $u_i$ is adjacent to precisely $k-5\ge 1$ of the vertices in
$\{w_{s}: s\notin\{i-1,i\}, 1\le s\le k-1\}$, so it is adjacent to at least one $w_{\ell}$ such that ${\ell}\notin\{i-1,i\}$.
If $i<\ell\le k-1$ then $u_ivu_{\ell}w_{\ell}u_i$ is a non-facial $4$-cycle (as $u_{i}w_{i}u_{i+1}\ldots w_{\ell-1}u_{\ell}$ lies in one of the regions bounded 
by this cycle while
$w_{\ell}u_{\ell+1}w_{\ell+1}\ldots w_{i-1}u_i$ lies in the other region). If $1\le \ell<i-1$ then $u_ivu_{\ell+1}w_{\ell}u_i$ is a non-facial $4$-cycle.
Since $Q$ can not have non-facial $4$ cycles by Lemma~\ref{la:optim}, this is a contradiction
\end{proof}

\section{Minimum Wiener Index of 5-Connected Triangulations} \label{mosaicsec}

Euler's formula shows that there are no $5$-connected triangulations  of order less than $12$.

First we state some facts about triangulations of a simple $n$-gon not using additional vertices. Triangulations of an $n$-gon can be 
viewed as planar graphs, where the outer face is bounded by an $n$ cycle
and all other faces are bounded by a $3$-cycle (we will refer to such faces as  {\sl triangles}). 
\begin{lemma}\label{la:ngon} Let $n\ge 4$. Any triangulation of a simple $n$-gon uses $n-3$  additional edges (i.e. edges which are not edges of the $n$-gon), and
has  at least $2$  triangles with exactly two of their boundary edges
on the $n$-gon. 
\end{lemma}

\begin{proof} The fact that the triangulation has $n-3$ edges (and consequently $n-2$ triangles) is  easy to prove by induction on $n$. When $n\ge 4$,
all these triangles have at most  $2$ boundary edges on the $n$-gon. As there are    $n-2$ triangles inside and $n$ edges on the $n$-gon itself, by the pigeonhole principle some two triangles must have two edges from edges of the $n$-gon.
\end{proof}

    We need the  following  basic facts about $5$-connected triangulations:

\begin{lemma}\label{la:basics}
Let $T$ be a $5$-connected triangulation of order $n$. The following are true:
 \begin{enumerate}[label={\upshape (\alph*)}]
 \item\label{case:cycles} Every 3-cycle is the boundary of a face and every 4-cycle is the boundary of a region whose interior
 does not contain vertices of the graph, and contains exactly one edge.
 \item\label{case:triangles}  Every edge lies on exactly  two triangles.
 If $abc$ and $bcd$ are triangles of $T$, then $ad$ is not an edge of $T$.
 \item\label{case:edgecycle} For every edge $xy$ of $T$,  there is precisely one $4$-cycle in $T$ that goes through its vertices, but does not use the $xy$ edge; hence the number of $4$-cycles in $T$ is $3(n-2)$.
 \item\label{case:nonedgecycle} If $x,y$ are non-adjacent vertices in $T$, then there is at most one $4$-cycle that contains them.
 \item\label{case:dist2} Let $D_i$ denote the number of unordered pairs of vertices at distance $i$ in $T$. We have $D_1=3(n-2)$ and $$D_2=\frac{1}{2}\sum_{x\in V(T)} d^2(x)-12(n-2).$$
 \item\label{case:lowerbound} $$W(T)\ge 3\binom{n}{2}+6(n-2)-\frac{1}{2}\sum_{x\in V(T)} d^2(x),$$
 with equality if and only if $T$ has diameter at most $3$.      
  \end{enumerate} 
  \end{lemma}

\begin{proof}
\ref{case:cycles}:  if $C$ is a cycle that separates two regions that both contain vertices in their interior, then the vertices of $C$ form a cutset, therefore $C$ has at least $5$ vertices.

\ref{case:triangles}: An edge bounds two faces that are triangles, and if there is a third 3-cycle using the edge, the other two edges of two of these 3-cycles give a $4$-cycle that
has vertices in both of its regions. If $abc$ and $bcd$ are triangles such that $ad$ is an edge, then one of $abd$, $acd$ 
 would be a non-facial triangle unless $n=4$. Both of these contradict ~\ref{case:cycles}, and ~\ref{case:triangles} follows.

\ref{case:edgecycle}: Since every $4$-cycle $abcd$ bounds a region  that has no vertices but has an edge (say $ac$), and if $ac$ is an edge then $bd$ cannot be an edge by \ref{case:triangles},
for every $4$-cycle there is a unique edge that is not part of the cycle and connects two of its vertices. So we can map
$4$-cycles to edges 
by assigning this edge to each cycle. This map is injective.  If two different $4$-cycles would map to the same edge, this edge is part of three triangles, contradicting \ref{case:triangles}.

As each edge lies on two triangles which together form a $4$ cycle, every edge is assigned to precisely one of these $4$-cycles, so the map is surjective as well. 
Thus, the number of $4$-cycles is the same as the number of edges, which is
$3(n-2)$ in any planar triangulation. ~\ref{case:edgecycle} follows.

\ref{case:nonedgecycle}: Assume $x,y$ are non-adjacent vertices that appear on two $4$-cycles. As each $4$-cycle containing $x,y$ has two $x$-$y$ paths of length $2$,
we have at least three $x$-$y$ paths of length $2$, say $xa_1y$, $xa_2y$, $xa_3y$. By~\ref{case:cycles}, for each $i,j\in\{1,2,3\}$, $i\ne j$ the cycle
$C_{ij}=xa_iya_jx$ bounds a region $\mathcal{R}_{ij}$ that contains no vertices in its interior.  But the three regions $\mathcal{R}_{12},\mathcal{R}_{13}\mathcal{R}_{23}$
together with their boundaries cover the entire plane, so $T$ has no other vertices besides $x,y,a_1,a_2,a_3$. This is a contradiction, as $5$-connected triangulations must have at least $12$ vertices.

\ref{case:dist2}:
Observe that $D_1$ is exactly the number of edges of $T$, $3(n-2)$. The formula
$\sum\binom{d(x)}{2}=\frac{1}{2}\sum d^2(x)-3(n-2)$ counts the number of paths of length $2$ between unordered pairs of vertices. If an unordered pair of vertices
has more than $1$ such path, it appears on a $4$-cycle, and by ~\ref{case:edgecycle} and ~\ref{case:nonedgecycle} this $4$-cycle is unique.
As each $4$-cycle contains exactly two such unordered pairs of vertices, the number of unordered pairs of vertices
that have a path of length $2$ between them is $\frac{1}{2}\sum d^2(x)-9(n-2)$ by ~\ref{case:edgecycle}. As every edge is contained in exactly one
$4$-cycle,
this equals $D_1+D_2$, proving ~\ref{case:dist2}.

\ref{case:lowerbound}:  As $\sum_i D_i=\binom{n}{2}$, we get 
\begin{eqnarray*}W(T)&=&\sum_i iD_i\ge D_1+2D_2+3\left(\binom{n}{2}-D_1-D_2\right)=3\binom{n}{2}-2D_1-D_2\\
&=&3\binom{n}{2}+6(n-2)-\frac{1}{2}\sum_{x\in V(T)}d^2(x),
\end{eqnarray*}
and equality holds precisely when the diameter is at most $3$.
\end{proof}

Analogously to  Section~\ref{sunfl},
we define an auxiliary drawn graph, which we will use extensively. Let $T$ be a $5$-connected triangulation, and let $v \in V(T)$ have degree $d$. We define the {\sl mosaic graph}  $M_v$  at vertex $v$,  together with its planar drawing, in the following way.    $M_v$   contains the neighbors of $v$ in $G$, $u_1,u_2,\ldots, u_d$, 
with the edges $vu_i$, such that vertices $u_i$ are
labeled according  the clockwise cyclic order  of the edges. We include the edges
$u_iu_{i+1}\in E(T)$  for every $1\le i\le d$ (indices  are taken modulo $d$) in $M_v$, following the drawing of $T$.  We also add a vertex $w_i\not=v$, which is a common
neighbor of $u_i$ and $u_{i+1}$, together with edges joining them to  $u_i$ and $u_{i+1}$ in $T$, following the drawing of $T$, for every $i$. 
We understand $M_v$ as a part of the drawing of $T$. We will show that $M_v$ has $2d+1$ distinct vertices. 

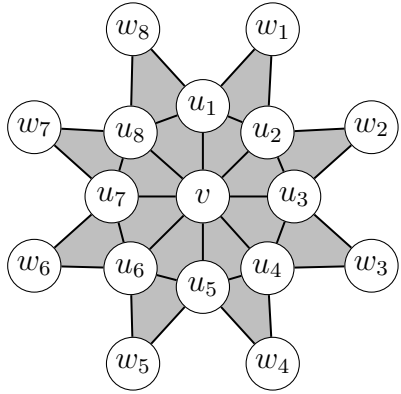
\begin{figure}[ht]
		\centering
			\begin{tikzpicture}
			[scale=0.6,xscale=1,inner sep=.5mm,
			vertex/.style={circle,draw,minimum size=7mm,fill=white},
			thickedge/.style={line width=0.75pt}] 
			\fill[gray!50] (0,2)--(1.53073, 3.69552)--(1.4142,1.4142)--(3.69552, 1.53073)--(2,0)--(3.69552, -1.53073)--(1.4142,-1.5858)--(1.53073, -3.69552)--(0,-2)-- (-1.53073, -3.69552)--(-1.5858,-1.5858)--(-3.69552, -1.53073)--(-2,0)--(-3.69552, 1.53073)--(-1.5858,1.4142)--(-1.53073, 3.69552)--(0,2);

			\node[vertex] (v) at (0,0) {$v$};
			\node[vertex] (u1) at (0,2) {$u_{1}$};
			\node[vertex] (u2) at (1.4142,1.4142) {$u_{2}$};
			\node[vertex] (u3) at (2,0) {$u_{3}$};
			\node[vertex] (u4) at (1.4142,-1.5858) {$u_{4}$};
			\node[vertex] (u5) at (0,-2) {$u_{5}$};
			\node[vertex] (u6) at (-1.5858,-1.5858) {$u_{6}$};
			\node[vertex] (u7) at (-2,0) {$u_{7}$};
			\node[vertex] (u8) at (-1.5858,1.4142) {$u_{8}$};
			
			\node[vertex] (w1) at (1.53073, 3.69552) {$w_{1}$};
			\node[vertex] (w2) at (3.69552, 1.53073) {$w_{2}$};
			\node[vertex] (w3) at (3.69552, -1.53073) {$w_{3}$};
			\node[vertex] (w4) at (1.53073, -3.69552) {$w_{4}$};
			\node[vertex] (w5) at (-1.53073, -3.69552) {$w_{5}$};
			\node[vertex] (w6) at (-3.69552, -1.53073) {$w_{6}$};
			\node[vertex] (w7) at (-3.69552, 1.53073) {$w_{7}$};
			\node[vertex] (w8) at (-1.53073, 3.69552) {$w_{8}$};
            
        	\draw[thickedge] (v)--(u1);
            \draw[thickedge] (v)--(u2);
            \draw[thickedge] (v)--(u3);
            \draw[thickedge] (v)--(u4);
            \draw[thickedge] (v)--(u5);
            \draw[thickedge] (v)--(u6);
            \draw[thickedge] (v)--(u7);
            \draw[thickedge] (v)--(u8);
            
            \draw[thickedge] (u1)--(w1)--(u2);
            \draw[thickedge] (u2)--(w2)--(u3);
            \draw[thickedge] (u3)--(w3)--(u4);
            \draw[thickedge] (u4)--(w4)--(u5);
            \draw[thickedge] (u5)--(w5)--(u6);
            \draw[thickedge] (u6)--(w6)--(u7);
            \draw[thickedge] (u7)--(w7)--(u8);
            \draw[thickedge] (u8)--(w8)--(u1);

            \draw[thickedge] (u1)--(u2)--(u3)--(u4)--(u5)--(u6)--(u7)--(u8)--(u1);    
                
			\end{tikzpicture}
	    \caption{The mosaic graph $M_v$ around $v$,  with $d=8$. The grey region is $\mathcal{R}_v$.}
	\label{fig:mosaic}
\end{figure}

\begin{lemma}\label{la:mosaic} If $T$ is a drawing of a $5$-connected triangulation of order $n$, and $v$ is any vertex of $T$ with  degree $d(v)=d$, the mosaic graph $M_v$  in $T$ has $2d+1$ distinct vertices. Furthermore, 
 the region ${\mathcal{R}}_v$ that is bounded by the cycle
$u_1w_1u_2w_2\ldots u_dw_d$ and contains the vertex $v$, contains edges and vertices from $T$ if and only if they are edges and vertices  in  the mosaic graph $M_v$. In addition, 
$T$ contains at least one vertex not in $M_v$, consequently $\Delta(T)\le \lfloor \frac{n}{2}\rfloor -1$. Moreover, $n\ne 13$.
\end{lemma}

\begin{proof}
Since $T$ is $5$-connected, $\delta(T)\geq 5$. 
As before,
label the neighbors of $v$ by $u_1, \dots, u_d$, in their planar clockwise cyclic order around $v$. We  get  for free that $u_iu_{i+1}$ is an edge in $T$, since we have a 
triangulation. 
For each pair of successive neighbors $u_i$ and $u_{i+1}$ (indices taken modulo $d$), let $w_i \not=v $ be their
common neighbor that completes the face 
that has $u_iu_{i+1}$ on its boundary, but not $v$.
 This means, in particular, 
that ${\mathcal{R}}_v$ will satisfy the required property, so we just need to show that the vertices listed in $M_v$ are all distinct.

If $y$ is a neighbor of $u_i$ and $y\notin\{v,w_{i-1},w_i, u_{i-1}, u_{i+1}\}$ then $y$ must lie between $w_{i-1}$ and $w_i$ in the planar cyclic order around $u_i$, In particular,
as $d(u_i)\ge 5$, we have that
$w_{i-1}\ne w_i$.

Also, $u_i\ne w_j$ for all $1\le i,j\le d$. For  $j\in\{i-1,i\}$ this is obvious, and for other values of $j$ if $u_i=w_j$ then  $v,u_i,u_j$ is a 3-element cutset.

Assume  now that  $w_i = w_j$ for some $j\ne i$. We already have that $j\notin\{i-1,i+1\}$ and hence the vertices $u_i,u_{i+1},u_j,u_{j+1}$ are all distinct.
We consider two regions of the planar drawing of $T$: ${\mathcal{R}}_1$ is bounded by the $4$-cycle $u_{i+1} v u_j w_i$ and does not contain the vertex
$u_i$, and ${\mathcal{R}}_2$ is bounded by the $4$-cycle $u_ivu_{j+1}w_i$ and does not contain the vertex $u_{i+1}$.

The neighbors of $u_i$ that differ from $v,u_{i+1}$ and $w_i$, lie in ${\mathcal{R}}_2$ and
the neighbors of $u_j$ that differ from $v,u_{j+1}$ and $w_i$, must lie in ${\mathcal{R}}_1$. Hence $\{v,u_{i+1},u_{j+1}, w_{i}\}$ separates $u_i$ from $u_{j}$ (see Figure ~\ref{fig:cutsetq2}), contradicting
that $T$ is $5$-connected.

	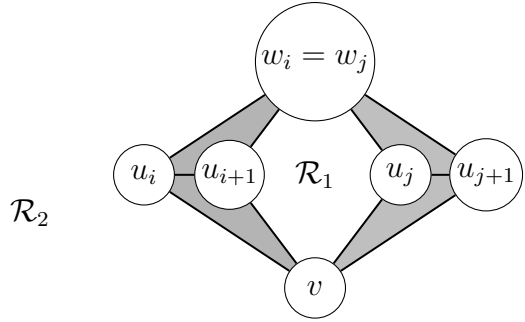
\begin{figure}[ht]
		\centering
			\begin{tikzpicture}
			[scale=0.5,xscale=0.75,inner sep=.5mm,
			vertex/.style={circle,draw,fill=white,minimum size=8mm},
			thickedge/.style={line width=0.75pt}] 
			\coordinate (V) at (5,0);
			\coordinate (W) at (5,6);
			\coordinate (U1) at (-1,3);
			\coordinate (U2) at (2,3);
			\coordinate (U3) at (8,3);
			\coordinate (U4) at (11,3);
			\fill[gray!59] (V)--(U1)--(W)--(U2)--cycle;
			\fill[gray!50] (V)--(U3)--(W)--(U4)--cycle;						
		
			\node at  (-5,2) {${\mathcal{R}}_2$};
			\node[vertex] (v) at (5,0) {$v$};
			\node[vertex] (ui) at (-1,3) {$u_{i}$};
			\node[vertex] (ui1) at (2,3) {$u_{i+1}$};
			\node[vertex] (uj) at (8,3) {$u_{j}$};
			\node[vertex] (uj1) at (11,3) {$u_{j+1}$};
			\node[vertex] (w) at (5,6) {$w_i=w_j$};
			\node at (5,3) {${\mathcal{R}}_1$};
			\draw[thickedge] (v)--(ui);
			\draw[thickedge] (v)--(ui1);
			\draw[thickedge] (v)--(uj);
			\draw[thickedge] (v)--(uj1);
			\draw[thickedge] (w)--(ui);
			\draw[thickedge] (w)--(ui1);
			\draw[thickedge] (w)--(uj);
			\draw[thickedge] (w)--(uj1);
			\draw[thickedge] (ui)--(ui1);
			\draw[thickedge] (uj)--(uj1);			
		\end{tikzpicture}
		\caption{4-element cutset $\{v,u_{i+1},u_{j+1},w_i\}$ in $M_v$ when $w_i=w_j$. The shaded regions are unions of faces, so they have no additional vertices.}
		\label{fig:cutsetq2}
	\end{figure}

Now let $v$ be a vertex  of $T$ with maximum degree, i.e.,  $d(v)=\Delta (T)=\Delta$. The mosaic graph $M_v$ around $v$ contains $2\Delta+1$ vertices.
If $T$ contains a vertex that is not in $M_v$, then $2\Delta+2\le n$, and  the claimed inequality follows.

If every vertex of $T$ is  in $M_v$, then set of edges 
$F$ not in $M_v$ form a triangulation of the $2\Delta$-cycle $u_1w_1u_2w_2\ldots u_{\Delta}w_{\Delta}u_1$ on the region different from ${\mathcal{R}}_v$; consequently
$|F|=2\Delta-3$.
Note that for  any $1\le i< j\le \Delta$, if $u_iu_j\in F$, then the 3-cycle $u_iu_jv$  separates the vertices
$w_i$ and $w_{j}$, so $u_i,u_j,v$   would be a cutset of size $3$, a contradiction.
If for  any $j\notin\{i,i-1\}$,  $u_iw_j\in F$, then the $4$-cycle $u_iw_ju_jv$ has the vertices $w_i$ and $w_{i-1}$ on its different sides, giving a cutset of
size $4$, which is also a contradiction. Therefore every edge in $F$ connects two vertices of $W=\{w_1,\ldots,w_{\Delta}\}$.

But then for every $i$,  the edges $w_{i-1}u_i$ and $u_iw_i$ lie on the boundary of the same face, giving
$w_{i-1}w_i\in F$. Hence $w_1,\ldots,w_{\Delta}$ determines a $\Delta$-gon (all of the sides are in $F$), and this $\Delta$-gon is triangulated by the 
remaining edges of $F$. Lemma~\ref{la:ngon} applies. Say, $w_{i-1},w_i, w_{i+1} $ is  a triangle with two edges on the boundary of the $\Delta$-gon. Then $d(w_i)=4$, contradicting the fact that $T$ is 
$5$-connected.

Finally, assume to the contrary that $T$ has $13$ vertices. Then $\Delta(T)\le 5$, therefore $T$ is $5$-regular. 
The sum of degrees of $T$ is odd, contradicting the Handshaking Lemma.
\end{proof}

	\begin{figure}[ht]
		\centering
			\begin{tikzpicture}
			[scale=0.6,yscale=1.1,inner sep=0.75mm,vertex/.style={circle,draw},thickedge/.style={line width=0.75pt}]
			\coordinate (A1) at (0,2.5);
			\coordinate (A2) at (1,2.5);
			\coordinate (A3) at (2,2.5);
			\coordinate (A4) at (3,2.5);
			\coordinate (A5) at (4,2.5);
			\coordinate (A6) at (5,2.5);
			\coordinate (B1) at (0,1.5);
			\coordinate (B2) at (1,1.5);
			\coordinate (B3) at (2,1.5);
			\coordinate (B4) at (3,1.5);
			\coordinate (B5) at (4,1.5);
			\coordinate (B6) at (5,1.5);		
			\fill[gray!20]  (B6) to[out=250,in=0] (2.3,-.6) to [out=180, in=240] (A1)--(B1)--(A2)--(B2)--(A3)--(B3)--(A4)--(B4)--(A5)--(B5)--(A6)--(B6);
			\node[vertex] (p1) at (2.5,4) [fill=black] {};
			\node[vertex] (p2) at (2.5,0) [fill=black] {};
			\node[vertex] (a1) at (0,2.5) [fill=black] {};
			\node[vertex] (a2) at (1,2.5) [fill=black] {};
			\node[vertex] (a3) at (2,2.5) [fill=black] {};
			\node[vertex] (a4) at (3,2.5) [fill=black] {};
			\node[vertex,dashed] (a5) at (4,2.5) [fill=gray] {};
			\node[vertex] (a6) at (5,2.5) [fill=black] {};
			\node[vertex] (b1) at (0,1.5) [fill=black] {};
			\node[vertex] (b2) at (1,1.5) [fill=black] {};
			\node[vertex] (b3) at (2,1.5) [fill=black] {};
			\node[vertex] (b4) at (3,1.5) [fill=black] {};
			\node[vertex,dashed] (b5) at (4,1.5) [fill=gray] {};
			\node[vertex] (b6) at (5,1.5) [fill=black] {};
			\draw[thickedge] (p1)--(a1);
			\draw[thickedge] (p1)--(a2);
			\draw[thickedge] (p1)--(a3);
			\draw[thickedge] (p1)--(a4);
			\draw[thickedge,dashed] (p1)--(a5);
			\draw[thickedge] (a1)--(a2)--(a3)--(a4);
			\draw[thickedge,dashed] (a4)--(a5);
			\draw[thickedge] (a5)--(a6);
			\draw[thickedge] (p1)--(a6);
			\draw[thickedge] (p2)--(b1);
			\draw[thickedge] (p2)--(b2);
			\draw[thickedge] (p2)--(b3);
			\draw[thickedge] (p2)--(b4);
			\draw[thickedge,dashed] (p2)--(b5);
			\draw[thickedge] (p2)--(b6);
			\draw[thickedge] (b1)--(b2)--(b3)--(b4);
			\draw[thickedge,dashed] (b4)--(b5);
			\draw[thickedge] (b5)--(b6);
			\draw[thickedge] (a1)--(b1);		
			\draw[thickedge] (a2)--(b2);
			\draw[thickedge] (a3)--(b3);
			\draw[thickedge] (a4)--(b4);
			\draw[thickedge,dashed] (a5)--(b5);
			\draw[thickedge] (a6)--(b6);
			\draw[thickedge] (a2)--(b1);
			\draw[thickedge] (a3)--(b2);
			\draw[thickedge] (a4)--(b3);
			\draw[thickedge,dashed] (a5)--(b4);
			\draw[thickedge] (a6)--(b5);
			\draw[thickedge] (b1) to[out=290,in=180] (2.5,-0.3);
			\draw[thickedge] (b6) to[out=250,in=0] (2.5,-0.3);
			\draw[thickedge] (a1) to[out=70,in=180] (2.5,4.3);
			\draw[thickedge] (a6) to[out=110,in=0] (2.5,4.3);
			\draw[thickedge] (b6) to[out=250,in=0] (2.3,-.6);
			\draw[thickedge] (a1) to[out=240,in=180] (2.3,-.6);
			\end{tikzpicture}
		    \qquad
			\begin{tikzpicture}
			[scale=0.6,yscale=-1,inner sep=0.75mm,vertex/.style={circle,draw},thickedge/.style={line width=0.75pt}]
			\coordinate(A1) at (0,3.5) ;
			\coordinate(A2) at (1,3.5) ;
			\coordinate (A3) at (2,3.5) ;
			\coordinate (A4) at (3,3.5) ;
			\coordinate (A5) at (4,3.5) ;
			\coordinate (A6) at (5,3.5) ;
			\coordinate (A7) at (6,3.5);
			\coordinate (B1) at (0,2.5) ;
			\coordinate (B3) at (3,2.5) ;
			\coordinate (B4) at (4,2.5) ;
			\coordinate (B5) at (5,2.5);
			\coordinate (B6) at (6,2.5) ;
			\coordinate (D1) at (1,2.5) ;
			\coordinate (D2) at (2,2.5) ;
			
			\fill[gray!20] (A7) to[out=80,in=0] (2.5,5.8) to[out=180,in=120] (B1)--(A1)--(D1)--(A2)--(D2)--(A3)--(B3)--(A4)--(B4)--(A5)--(B5)--(A6)--(B6)--(A7);

			\node[vertex] (p1) at (3,5) [fill=black] {};
			\node[vertex] (p2) at (3,1) [fill=black] {};	
			\node[vertex] (a1) at (0,3.5) [fill=black] {};
			\node[vertex] (a2) at (1,3.5) [fill=black] {};
			\node[vertex] (a3) at (2,3.5) [fill=black] {};
			\node[vertex] (a4) at (3,3.5) [fill=black] {};
			\node[vertex,dashed] (a5) at (4,3.5) [fill=gray] {};
			\node[vertex] (a6) at (5,3.5) [fill=black] {};
			\node[vertex] (a7) at (6,3.5) [fill=black] {};
			\node[vertex] (b1) at (0,2.5) [fill=black] {};
			\node[vertex] (b2) at (2,1.9) [fill=black] {};
			\node[vertex] (b3) at (3,2.5) [fill=black] {};
			\node[vertex] (b4) at (4,2.5) [fill=black] {};
			\node[vertex,dashed] (b5) at (5,2.5) [fill=gray] {};
			\node[vertex] (b6) at (6,2.5) [fill=black] {};
			\node[vertex] (d1) at (1,2.5) [fill=black] {};
			\node[vertex] (d2) at (2,2.5) [fill=black] {};
			\draw[thickedge] (p1)--(a1);
			\draw[thickedge] (p1)--(a2);
			\draw[thickedge] (p1)--(a3);
			\draw[thickedge] (p1)--(a4);
            \draw[thickedge,dashed] (p1)--(a5);
            \draw[thickedge] (p1)--(a6);
            \draw[thickedge] (p1)--(a7);
			\draw[thickedge] (a1)--(a2)--(a3)--(a4);
			\draw[thickedge,dashed] (a4)--(a5);
			\draw[thickedge] (a5)--(a6)--(a7);
			\draw[thickedge] (a5)--(a6);
			\draw[thickedge] (p1)--(a6);
			\draw[thickedge] (p2)--(b1);
			\draw[thickedge] (p2)--(b2);
			\draw[thickedge] (p2)--(b3);
			\draw[thickedge] (p2)--(b4);
			\draw[thickedge,dashed] (p2)--(b5);
			\draw[thickedge] (p2)--(b6);
			\draw[thickedge] (p2)--(b6);
			\draw[thickedge] (b1)--(b2)--(b3)--(b4);
			\draw[thickedge,dashed] (b4)--(b5);
			\draw[thickedge] (b5)--(b6);
			\draw[thickedge] (b5)--(b6);
			\draw[thickedge] (a1) to[out=90,in=180] (3,5.3);
			\draw[thickedge] (3,5.3) to[out=0,in=90] (a7);
            \draw[thickedge] (b1) to[out=290,in=180] (3,0.6);
			\draw[thickedge] (3,0.6) to[out=0,in=250] (b6);
           \draw[thickedge] (a7) to[out=80,in=0] (2.5,5.8);
			\draw[thickedge] (2.5,5.8) to[out=180,in=120] (b1);	
            \draw[thickedge] (d1)--(a1);
			\draw[thickedge] (d1)--(a2);
			\draw[thickedge] (d1)--(b1);
			\draw[thickedge] (d1)--(b2);
			\draw[thickedge] (d1)--(d2);
			\draw[thickedge] (d2)--(a2);
			\draw[thickedge] (d2)--(a3);
			\draw[thickedge] (d2)--(b2);
			\draw[thickedge] (d2)--(b3);
			\draw[thickedge] (a1)--(b1);
			\draw[thickedge] (a3)--(b3);
			\draw[thickedge] (a4)--(b4);
			\draw[thickedge,dashed] (a5)--(b5);
			\draw[thickedge] (a4)--(b3);
			\draw[thickedge,dashed] (a5)--(b4);
			\draw[thickedge] (a6)--(b5);
			\draw[thickedge] (a6)--(b6);
			\draw[thickedge] (a7)--(b6);
			\end{tikzpicture}
			\caption{The triangulation $ T_n^5 $, which minimizes the Wiener index among all $5$-connected triangulations of order $n=2k\geq 12$  (left) 
			and of order $ n=2k+1 \geq 15$  (right). Gray vertices and dashed edges indicate the pattern to be repeated. The shaded region shows the mosaic graph around a degree $\lfloor\frac{n}{2}\rfloor-1$ vertex.} 
			\label{fig:t5cminwi}
    \end{figure}
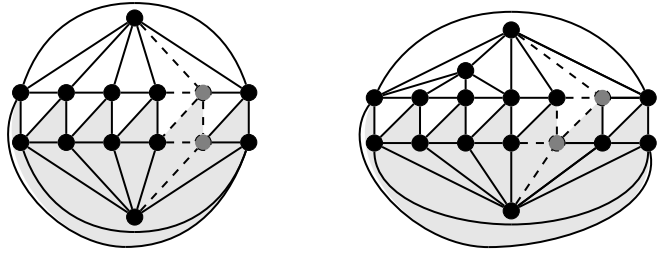

	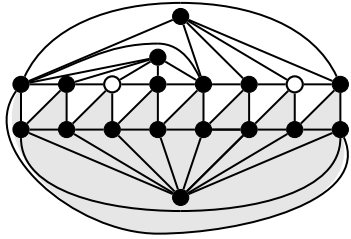
\begin{figure}[ht]
		\centering
			\begin{tikzpicture}
			[scale=0.6,yscale=-1,inner sep=0.75mm,vertex/.style={circle,draw},thickedge/.style={line width=0.75pt}]
			\coordinate(A1) at (0,3.5) ;
			\coordinate(A2) at (1,3.5) ;
			\coordinate (A3) at (2,3.5) ;
			\coordinate (A4) at (3,3.5) ;
			\coordinate (A5) at (4,3.5) ;
			\coordinate (A6) at (5,3.5) ;
			\coordinate (A7) at (6,3.5);
			\coordinate (A8) at (7,3.5);		
			\coordinate (B1) at (0,2.5) ;
			\coordinate (B3) at (3,2.5) ;
			\coordinate (B4) at (4,2.5) ;
			\coordinate (B5) at (5,2.5);
			\coordinate (B6) at (6,2.5) ;
			\coordinate (B7) at (7,2.5) ;			
			\coordinate (D1) at (1,2.5) ;
			\coordinate (D2) at (2,2.5) ;			
			\fill[gray!20] (A8) to[out=-290,in=0] (3,5.8) to[out=180,in=-240] (B1)--(A1)--(D1)--(A2)--(D2)--(A3)--(B3)--(A4)--(B4)--(A5)--(B5)--(A6)--(B6)--(A7)--(B7)--(A8);			
			\node[vertex] (p1) at (3.5,5) [fill=black] {};
			\node[vertex] (p2) at (3.5,1) [fill=black] {};				
			\node[vertex] (a1) at (0,3.5) [fill=black] {};
			\node[vertex] (a2) at (1,3.5) [fill=black] {};
			\node[vertex] (a3) at (2,3.5) [fill=black] {};
			\node[vertex] (a4) at (3,3.5) [fill=black] {};
			\node[vertex] (a5) at (4,3.5) [fill=black] {};
			\node[vertex] (a6) at (5,3.5) [fill=black] {};
			\node[vertex] (a7) at (6,3.5) [fill=black] {};
			\node[vertex] (a8) at (7,3.5) [fill=black] {};			
			\node[vertex] (b1) at (0,2.5) [fill=black] {};
			\node[vertex] (b2) at (3,1.9) [fill=black] {}; 
			\node[vertex] (b3) at (3,2.5) [fill=black] {};
			\node[vertex] (b4) at (4,2.5) [fill=black] {};
			\node[vertex] (b5) at (5,2.5) [fill=black] {};
			\node[vertex,thick] (b6) at (6,2.5) [fill=white] {};
			\node[vertex] (b7) at (7,2.5) [fill=black] {};
			\node[vertex] (d1) at (1,2.5) [fill=black] {};
			\node[vertex,thick] (d2) at (2,2.5) [fill=white] {};
			\draw[thickedge] (p1)--(a1);
			\draw[thickedge] (p1)--(a2);
			\draw[thickedge] (p1)--(a3);
			\draw[thickedge] (p1)--(a4);
			\draw[thickedge] (p1)--(a8);			
            \draw[thickedge] (p1)--(a5);
            \draw[thickedge] (p1)--(a6);
            \draw[thickedge] (p1)--(a7);
			\draw[thickedge] (a1)--(a2)--(a3)--(a4);
			\draw[thickedge] (a4)--(a5);
			\draw[thickedge] (a5)--(a6);
			\draw[thickedge] (a6)--(a7);
			\draw[thickedge] (a5)--(a6);
			\draw[thickedge] (p1)--(a6);
			\draw[thickedge] (p2)--(b1);
			\draw[thickedge] (b4)--(b2); 
			\draw[thickedge] (b1) to[out=335,in=180] (3,1.6);  
			\draw[thickedge] (3,1.6) to[out=0,in=255] (b4);   	
			\draw[thickedge] (p2)--(b4);
			\draw[thickedge] (p2)--(b5);
			\draw[thickedge] (p2)--(b6);
			\draw[thickedge] (p2)--(b7);
			\draw[thickedge] (b1)--(b2)--(b3)--(b4);
			\draw[thickedge] (b4)--(b5);
			\draw[thickedge] (b5)--(b6);
			\draw[thickedge] (b6)--(b7);
			\draw[thickedge] (a1) to[out=90,in=180] (3.5,5.3);
			\draw[thickedge] (3.5,5.3) to[out=0,in=90] (a8);
            \draw[thickedge] (b1) to[out=290,in=180] (3.5,.7);
			\draw[thickedge] (3.5,.7) to[out=0,in=250] (b7);
            \draw[thickedge] (b1) to[out=-240,in=180] (3,5.8) to[out=0,in=-290] (a8);	
            \draw[thickedge] (d1)--(a1);
			\draw[thickedge] (d1)--(a2);
			\draw[thickedge] (d1)--(b1);
			\draw[thickedge] (d1)--(b2);
			\draw[thickedge] (d1)--(d2);
			\draw[thickedge] (d2)--(a2);
			\draw[thickedge] (d2)--(a3);
			\draw[thickedge] (d2)--(b2); 
			\draw[thickedge] (d2)--(b3);
			\draw[thickedge] (a1)--(b1);
			\draw[thickedge] (a3)--(b3);
			\draw[thickedge] (a4)--(b4);
			\draw[thickedge] (a5)--(b5);
			\draw[thickedge] (a4)--(b3);
			\draw[thickedge] (a5)--(b4);
			\draw[thickedge] (a6)--(b5);
			\draw[thickedge] (a6)--(b6);
			\draw[thickedge] (a7)--(b6);
			\draw[thickedge] (a8)--(b7);
			\draw[thickedge] (a8)--(a7);
			\draw[thickedge] (a7)--(b7);	
			\end{tikzpicture}
			\caption{The $5$-connected triangulation $ X $. The two white vertices are at distance $4$. The shaded region shows the mosaic graph around the degree $8$ vertex.} 
			\label{fig:t5cotherodd1}
    \end{figure}

For every $n\ge 12$, $n\ne 13$, the $n$-vertex triangulation $T_n^5$ is defined by Figure ~\ref{fig:t5cminwi} (these will be our minimizers of the Wiener index).
The following lemma is easy to verify and we leave the details to the reader.
\begin{lemma} \label{Tn5lista}
Assume that $n\ge 12$, $n\ne 13$.
\begin{enumerate}[label={\upshape (\alph*)}]
\item $T_n^5$ is a $5$-connected triangulation.
\item $T_n^5$ has diameter 3.
\item \label{tevendegseq} For $n$ even, the degree sequence of $T_{n}^5$ is $\frac{n}{2}-1, \frac{n}{2}-1, 5,\ldots,5$.
\item  \label{todddegseq} For $n$ odd, the degree sequence of $T_{n}^5$ is $\lfloor\frac{n}{2}\rfloor-1,\lfloor\frac{n}{2}\rfloor-2,6,6,5,\ldots, 5$. (For $n=15$, the terms in this sequence are {\sl not} in decreasing order.)
\item \label{eformula}  \begin{equation*} 
    W(T_n^{5}) =
    \begin{cases}
    \frac{5n^2}{4} -7n + 12, & \text{if $n$ is even}, \\
    \frac{5n^2}{4} -6n - \frac{9}{4}, & \text{if $n$ is odd}. \\
\end{cases}
 \end{equation*}
\end{enumerate}
\end{lemma} 
The $5$-connected triangulation $X$ of order $19$, defined by
  Figure ~\ref{fig:t5cotherodd1},
has
\begin{equation} \label{indexX}
W(X)=335=W(T_{19}^5).
\end{equation}
We will also show that $X$ is the {\sl only} $5$-connected triangulation that is not isomorphic to 
any $T_n^5$ and achieves the minimum Wiener index for its order. Note that as
$X$ is of diameter $4$, the lower bound in Lemma~\ref{la:basics}~\ref{case:lowerbound} cannot be used to compute $W(X)$. The different diameter, and also the different degree sequence, implies that 
$X\not\simeq T_{19}^5$.

We define the {\sl extended mosaic graph} $M_v^{\star}$ by adding edges to $M_v$. Given a $5$-connected triangulation $G$ and a vertex $v$ with mosaic graph $M_v$,
 we introduce the graph $M_v^{\star}$, on the vertex set of $M_v$, by setting
 $$E(M_v^{\star})=E(M_v)\cup\{w_iw_{i+1}: 1\le i\le d(v), w_iw_{i+1}\in E(G)\}.$$ 
 Note that $w_iw_{i+1}\in E(G)$ if an only if
 $d(u_{i+1})=5$.
Let ${\mathcal{R}}_v^{\star}$ denote the extension of
 ${\mathcal{R}}_v$ by adding to it the faces bounded by the $3$-cycles $u_{i+1}w_iw_{i+1}$ for all edges $w_iw_{i+1}\in E(G)$;   let $C_v$ denote the 
 boundary cycle  of 
 ${\mathcal{R}}_v^{\star}$ and let ${\mathcal{Q}}_v^{\star}$ denote the other domain defined by the cycle $C_v$.
 Now all vertices of $G$ that are not vertices of $M_v$ and all edges of $E(G)\setminus E(M_v^{\star})$ lie in the region ${\mathcal{Q}}_v^{\star}$ of the drawing of $G$.

We will use the following notation in the rest of the section. Given a $5$-connected triangulation $G$ and a vertex $v$, if $a,b$ are vertices of $C_v$, then
$P_v(a,b)$ denotes the path on the cycle $C_v$ from $a$ to $b$ that follows the  clockwise cyclic order.
(So if $C_v=(w_1,w_2,\ldots,w_{d})$ in clockwise cyclic order, then $P_v(w_1,w_2)$ is just the edge $w_1w_2$ with its endpoints, while $P_v(w_2,w_1)$ goes through all vertices of the cycle and misses only the edge $w_1w_2$.)

\begin{lemma}\label{la:mvs} Let $G$ be a $5$-connected triangulation of order  $n\ge 12$, 
and let 
$v$ be a vertex of $G$ with $d(v)=d$.
 Consider the extended mosaic graph $M_v^{\star}$. The following are true:
\begin{enumerate}[label={\upshape(\alph*)}]
 \item\label{case:edgesout}
Every vertex $z\in V(C_v)$ has an edge of $E(G)\setminus E(M_v^{\star})$ incident upon it. 
 \item\label{case:cutedge}   If for some $z_1,z_2\in V(C_v)$ we have
 $z_1z_2\in E(G)\setminus E(M_v^{\star})$, then $z_1z_2$ cuts  ${\mathcal{Q}}_v^{\star}$ into two subregions, each containing a vertex of $G$ in its interior,
 and $z_1,z_2\in\{w_1,w_2\ldots,w_d\}$.
 \item\label{case:wearedone} If $n$ is even and $d(v)=\frac{n}{2}-1$, then $G\simeq T_n^5$.
\end{enumerate}
\end{lemma}

\begin{proof}
Set $W=\{w_1,\ldots,w_{d}\}$ and $U=\{u_1,\ldots,u_{d}\}$.

\ref{case:edgesout}: Observe that $W\subseteq V(C_v)\subseteq W\cup U$.
Vertices in $W$ have degree at most $4$ in $M_v^{\star}$, and vertices of $U$ have degree $5$ in $M_v^{\star}$. If a vertex of $U$ has degree $5$ in $G$, then
it is not a vertex of $C_v$.  \ref{case:edgesout} follows.

\ref{case:cutedge}: Let$z_1,z_2\in V(C_v)$ where
 $z_1z_2\in E(G)\setminus E(M_v^{\star})$. Assume first that $z_1z_2$ cuts  $Q_v^{\star}$ into two subregions, one of which (say ${\mathcal{R}}$) contains no vertices in its interior. 
 We will show that $\mathcal{R}$ contains a (triangular) face $f$ such that the boundary of $f$ has two edges $e_1,e_2$ that are on the boundary of $\mathcal{R}$ and $z_1z_2\notin\{e_1,e_2\}$.
 This is obviously true when the boundary of $\mathcal{R}$ is a $3$-cycle.
 Otherwise the edges lying in the interior of ${\mathcal{R}}$ are edges of $E(G)\setminus E(M_v^{\star})$ that form a triangulation of
 ${\mathcal{R}}$, and by Lemma~\ref{la:ngon} this triangulation contains two faces with two boundary edges on the boundary of ${\mathcal{R}}$.
 One of these faces, $f$, does not have the edge $z_1z_2$ on its boundary.
Let $c$ be the common endvertex of the two edges $e_1,e_2$ on the boundary of ${\mathcal{R}}$. Then $c\notin\{z_1,z_2\}$ and $c$ cannot have an edge from $E(G)\setminus E(M_v^{\star})$ incident upon it, contradicting \ref{case:edgesout}. So $z_1z_2$ cuts $Q_v^{\star}$ into two subregions, both of which contains a vertex in its interior. Now assume to the contrary that $\{z_1,z_2\}\cap U\ne\emptyset$. We may assume that $z_1=u_1$. Then $z_2=u_{\ell}$ for some $3\le \ell\le d-1$ 
or $z_2=w_j$ for some $3\le j\le d-2$ ($j\ne 2$ and $j\ne d-1$, using Lemma~\ref{la:basics}~\ref{case:triangles} for edges $u_1u_2$ and $u_1u_d$). 
If $z_2=u_{\ell}$, then $u_1vu_{\ell}$ is a separating $3$-cycle (as $w_2$ and $w_d$ are in different regions of this cycle),
and if $z_2=w_j$ then $u_1w_ju_jv$ is a separating $4$-cycle (as $w_2$ and $w_d$  are in different regions); both of which contradict the $5$-connectedness of $G$.

\ref{case:wearedone}:
Assume now that $n$ is even and $d(v)=\frac{n}{2}-1$. Lemma~\ref{la:mosaic} gives $\Delta(G)=d(v)$. $G$ has exactly one vertex, say $x$, not in $M_v$,
and hence
in the region ${\mathcal{Q}}_v^{\star}$. Then the already proven parts  \ref{case:edgesout} and \ref{case:cutedge} imply that $E(G)\setminus E(M_v^{\star})=
\{z_1x:z_1\in V(C_v)\}$. As $W\subseteq C_v$,  $d(x)\ge |W|=d(v)=\Delta(G)$, we get $d(x)=\frac{n}{2}-1$ and $W=C_v$, and each edge of the form
$w_iw_{i+1}$ is an edge of $M_v^{\star}$. \ref{case:wearedone} follows.
\end{proof}

\begin{theorem} \label{t5cminwieven} Assume that $n\geq 12$ and $n$ is even. 
The triangulation  $ T_n^5 $, which was defined in Figure~\ref{fig:t5cminwi},  is the unique minimizer of the Wiener index among all $5$-connected triangulations  of order $n$. \end{theorem}
\noindent

\begin{proof} 
Let $n\ge 12$ be even and assume
$T$ is a $5$-connected triangulation on $n=2k$ vertices ($k\ge 6$). The  degree sum of $T$ is $2(3n-6)=6n-12$, and 
Lemma~\ref{la:mosaic} gives $\Delta(T)\leq \frac{n}{2}-1$. By Lemma~\ref{la:proc} the integer sequence $y_1,\ldots, y_n$ that sums to $6n-12$, satisfies $5\le y_i\le \frac{n}{2}-1$ and has the largest sum of squares
is the sequence
$\frac{n}{2}-1,\frac{n}{2}-1,5,5,\ldots,5$, which is exactly the degree sequence of $T_n^5$ by Lemma~\ref{Tn5lista}~\ref{tevendegseq}. As $T_n^5$ has diameter $3$, by Lemma~\ref{la:basics}~\ref{case:lowerbound}
$T_n^5$ indeed has the minimum Wiener index among all $5$-connected $n$-vertex triangulations. We know that the degree sequence of $T$ is the same as
the degree sequence of $T_n^5$, so $T\simeq T_n^5$ follows from Lemma~\ref{la:mvs}~\ref{case:wearedone}.
\end{proof}

To characterize extremal 
triangulations
of odd order, we need a bit more information about their structure.

\begin{lemma}\label{la:nosuch}
There are no $5$-connected triangulations on $21$ vertices with degree sequence $8, 8, 8, 5, ... , 5$.
\end{lemma}

\begin{proof} 
Assume that $G$ is a $5$-connected triangulation on $21$ vertices with degree sequence $8,8,8,5,\ldots 5$. 
Let $v$ be a degree $8$ vertex, and let $x_1,x_2,x_3,x_4$ be the vertices in
$V(G)\setminus V(M_v^{\star})$. Set $X=\{x_1,x_2,x_3,x_4\}$, $U=\{u_i: 1\le i\le 8\}$ and
$W=\{w_i:1\le i\le 8\}$. 
Let $b$ be the number of connected components of the subgraph of $G$ induced by $X$ and $D_i$ be the component containig $x_i$, $c_i=|N(x_i)\cap X|$ and $\chi$ be the number of vertices of degree $8$ in $X$. 

Clearly $W\subseteq V(C_v)$, and by Lemma~\ref{la:mvs}~\ref{case:edgesout} and \ref{case:cutedge} all vertices of $V(C_v)\cap U$ have degree $8$ (consequently $|V(C_v)\cap U|\le 2$). Assume that $z\in V(C_v)\cap U$; then $|N(z)\cap X|=3$. Let $\{x_i,x_j,x_k\}= N(z)\cap X$, then 
without loss of generality $x_ix_jx_k$ is a path in $G$ whose edges form faces with the edges $x_iz,x_jz,x_kz$. Moreover, if $e,f$ are the two edges on $C_v$ that are incident upon $z$, then the cyclic order of the edges that lie in or on the boundary of $\mathcal{Q}_v^{\star}$ around $z$ is $e,zx_i,zx_j,zx_k,f$ or  $f,zx_i,zx_j,zx_k,e$; otherwise
one of the triangles $zx_ix_j$ or $zx_jx_k$ is a separating triangle, which is a contradiction. Finally, $x_ix_k\notin E(G)$, as otherwise $zx_ix_k$ is a separating triangle.

Assume first that $U\cap V(C_v)\ne\emptyset$; without loss of generality $u_1\in U\cap V(C_v)$, $u_1$ is 
adjacent to $x_2,x_3,x_4$ and $x_2x_3x_4$ is a path in $G$, consequently $x_2x_4\notin E(G)$.
We have that either $b=2$ and $D_1=\{x_1\}$, or $b=1$.
Without loss of generality we may assume that the cyclic order of edges around $u_1$ in $\mathcal{Q}_v^{\star}$ is $u_1w_1, u_1x_2,u_1x_3,u_1x_4,u_1w_8$. As $u_1w_1,u_1x_2$ (and also $u_1x_4,u_1w_8$) bound a common face, we have
$w_1x_2,w_8x_4\in E(G)$. Let $\mathcal{P}^{\star}$ be the region we get if we leave out from $\mathcal{Q}^{\star}_v$ the faces with $u_1$ on their boundary.

Consider the case when $|U\cap V(C_v)|=2$, i.e. for some $j\ne 1$ the vertex $u_j$ also has degree $8$. If $N(u_j)\cap X=N(u_1)\cap X=\{x_2,x_3,x_4\}$ (which must happen when $b=2$), we have that the path $x_2x_3x_4$ is induced in $D_2$. But then $u_1x_2u_jx_4$ is a separating $4$-cycle, which is a contradiction. Therefore we must have $b=1$, and without loss of generality for some $t\in\{2,3\}$,
$N(u_j)\cap X=\{x_1,x_t,x_{t+1}\}$ where $x_1x_t\in E(G)$ and $x_1x_{t+1}\notin E(G)$. 
This gives $|E(X)|\ge 3$, $|E(X,U)|=6$, and
(since the vertices of $X$ have degree $5$) $|E(X,W)|=20-6-2|E(X)|\le 8$. On the other hand, since $b=1$, Lemma~\ref{la:mvs}~\ref{case:cutedge} implies
that the set of edges in $E(G)\setminus E(M_v^{\star})$ that are incident upon $W$ form the set $E(X,W)$, and (as $w_1$ and $w_8$ have degree $3$ in $M_v^{\star}$)
consequently $|E(X,W)|\ge 10$, a contradiction. Therefore we must have $|U\cap V(C_v)|<2$, i.e. $U\cap V(C_v)=\{u_1\}$. 

Now we have that $U\cap V(C_v)=\{u_1\}$.
Assume that $b=2$, so $D_1=\{x_1\}$. 
Let $s$ and $t$ be the smallest and largest indices such that $w_s,w_t\in N(x_1)$;
$1\le s<s+4\le t\le 8$. The path $w_sx_1w_t$ cuts the region
$\mathcal{Q}_v^{\star}$ into two regions $\mathcal{Q}_1$ and $\mathcal{Q}_2$, where $\mathcal{Q}_1$ has $u_1$ on its boundary. Therefore $x_2,x_3,x_4$ lie inside $\mathcal{Q}_1$, 
$w_sw_t\in E(G)\setminus E(M_v^{\star})$, and by Lemma~\ref{la:mvs}~\ref{case:cutedge} for each $i:s\le i\le t$ $x_1w_i\in E(G)$. 
If $s\ne 1$ then $\{w_{s-1},w_{s+1},u_s,u_{s+1},w_t,x_1\}\subseteq N(w_s)$, so $w_s$ must have degree $8$. If $s=1$, then $\{ u_1, u_2, w_2, x_2, w_t, x_1\}\subseteq N(w_s)$, 
so $w_s$ has degree $8$. Similar arguments imply that $w_t$ also has degree $8$. But then $u_1,v,w_s,w_t$ all have degree $8$, a contradiction.
So we must have $b=1$, i.e. $x_1$ is connected to at least one other vertex in $X$. If $x_1$ is connected to both $x_2$ and $x_4$, then (as the edges $x_1x_2$ and $x_1x_4$ lie in
$\mathcal{P}^{\star}$) $u_1x_2x_1x_4$ is a separating $4$-cycle, which is a contradiction. We may assume $x_1x_4\notin E(G)$, so $c_1\in\{1,2\}$.
Then  $|E(X,W)|=13+3\chi-2c_1$. Since the number of degree $8$ vertices in $W$ is $1-\chi$,
$|E(X,W)|=10+3(1-\chi)=13-3\chi$. This gives $2c_1=6\chi$, so $3$ divides $c_1$, which is a contradiction. Thus we must have $U\cap V(C_v)=\emptyset$.

Therefore $W=V(C_v)$ and every vertex in $W$ has degree $4$ in $M_v^{\star}$, so every vertex in $W$ has either one or $4$ edges incident upon it from $E(G)\setminus E(M_v^{\star})$.
Set $F=E(W)\setminus E(M_v^{\star})$ and $m_x=|E(X)|$.
We have $2|F|+|E(W,X)|=\sum_{z\in W}(d(z)-4)=14-3\chi$ and $|E(X,W)|=(\sum_{z\in X}d(z))-2m_x=20+3\chi-2m_x$.

If $\chi=2$, then all vertices of $W$ have at most one neighbor in $X$. Since the two vertices in $X$ that have degree $8$ have at least $5$ neighbors in $W$, this implies that $5+5\le |W|=8$, a contradiction. Therefore $\chi\in\{0,1\}$; at least one vertex in $W$ has degree $8$.

Suppose $b=1$. Then $3\le m_x\le 5$, and $ |E(X,W)|=20-2m_x+3\chi$. On the other hand by Lemma~\ref{la:mvs} ~\ref{case:cutedge} $F=\emptyset$, so
$|E(X,W)|=14-3\chi$. This gives $m_x=3+3\chi$, consequently $\chi=0$, $m_x=3$, and exactly two vertices (say $w_{\ell},w_q$) in $W$ have degree $8$. But then at least one of the $4$-cycles of the form $w_{\ell}x_iw_qx_jw_{\ell}$ is separating, which is a contradiction. Thus we must have $b>1$. 

Since $b>1$, we must have either that the components spanned by $X$ are a $K_1$ and a $K_3$, or
the subgraph generated by $X$ has exactly $4-b$ edges (as all of its components are trees). 
In the former case $|E(X,W)|=14+3\chi$, in the latter $|E(X,W)|=12+2b+3\chi\ge16+3\chi>14-3\chi\ge |E(X,W)|$, which is a contradiction. 
Therefore the components spanned by $X$ are a $K_1$ and a $K_3$. We get that $2|F|+14+3\chi= 14-3\chi$, which gives $\chi=0$, and $F=\emptyset$.
So exactly two vertices (say $w_{\ell},w_q$) in $W$ have degree $8$, and $E(W,V(G))\setminus E(M_v^{\star})=E(X,W)$.
Without loss of generality the $K_3$ in $X$ is formed by the vertices $x_2,x_3,x_4$.
But then $X\subseteq N(w_{\ell})$, so the subgraph generated by  $\{x_2,x_3,x_4,w_{\ell}\}$ is a $K_4$.
This is a contradiction, as one of the triangles $w_{\ell}x_2x_3$, $w_{\ell}x_3x_4$, $w_{\ell}x_2x_4$ is separating, contradicting the $5$-connectedness of $G$.
\end{proof}

\begin{lemma}\label{la:mvsdeg} Let $n\ge 15$ be odd, and let $G$ be a $5$-connected triangulation of order $n$ with degree sequence $d_1\ge d_2\ge d_3\ge\ldots\ge d_n$. If
$W(G)\le W(T_n^5)$, then $d_1=\lfloor\frac{n}{2}\rfloor-1$, and one of the following holds:
 \begin{enumerate}[label={\upshape (\alph*)}]
 \item\label{case:k11} $n=23$ and the degree sequence of $G$ is $10,8,8,5,\ldots,5$
\item\label{case:main} $d_2\ge \lfloor\frac{n}{2}\rfloor-2$, $d_3+d_4\le 12$, and consequently $d_3\le7$, $d_4\le 6$ and $d_5=5$.
\end{enumerate}
\end{lemma}

\begin{proof} Set $n=2k+1$, then $k=\lfloor\frac{n}{2}\rfloor\ge 7$.
As $d_1\le k-1$, the only possible degree sequence for $k=7$ is $(6,6,6,5,\ldots,5)$, which satisfies the conclusion. Hence we may assume $k\ge 8$.
As $T_n^5$ has diameter $3$, by Lemma~\ref{la:basics}~\ref{case:lowerbound} and Lemma~\ref{Tn5lista}~\ref{todddegseq} we must have
$$\sum_{x\in V(G)}d^2(x)\ge \sum_{x\in V(T_n^5)}d^2(x)=(k-1)^2+(k-2)^2+2\cdot 6^2+5^2(2k-3)=2k^2+44k+2.$$

Assume first that $d_1\le k-2$.

If $k=8$, the only sequence possible is $6,6,6,6,6,5,\ldots,5$ whose sum of squares is $480<2\cdot 8^2+44\cdot 8+2$, which is a contradiction.

If $k=9$,  then
$$\sum_{x\in V(G)}d^2(x)\le 3\cdot 7^2+6^2+15\cdot 5^2=558<560=2\cdot 9^2+44\cdot9+2,$$
a contradiction.

If $k=10$, then By Lemma~\ref{la:nosuch} 
$$\sum_{x\in V(G)}d^2(x)\le 2\cdot 8^2+ 7^2+6^2+17\cdot 5^2=638<642=2\cdot 10^2+44\cdot10+2,$$
a contradiction.

If $k\ge 11$, then
$$\sum_{x\in V(G)}d^2(x)\le 2(k-2)^2+8^2+5^2(2k-2)=2k^2+42k+22<2k^2+44k+2,$$
a contradiction. So we proved that $d_1=k-1$.
If $k=8$, the only sequences possible are $7,7,6,5,\ldots,5$ and   $7,6,6,6,5\ldots,5$, which satisfy the conclusion. Hence we may assume $k\ge 9$.

Assume next that $d_2\le k-3$.
If $k=9$, the only sequence possible is $8,6,6,6,6,5,\ldots,5$ with degree square sum  $558<2\cdot 9^2 +44\cdot 9+2$.  
If $k=10$, 
$$\sum_{x\in V(G)}d^2(x)\le 9^2+2\cdot 7^2+6^2+17\cdot 5^2=640<2\cdot 10^2+44\cdot 10+2.$$
If $k\ge 11$, then
$$\sum_{x\in V(G)}d^2(x)\le (k-1)^2+(k-3)^2+8^2+5^2(2k-2)=2k^2+42k+24\le 2k^2+44k+2.$$
with equality only if $k=11$ and the degree sequence is $10,8,8,5,\ldots,5$, i.e. $G$ satisfies \ref{case:k11}.
Therefore we may assume that $d_2\ge k-2$.

As we already know that  $d_1= k-1$ and $d_2\ge k-2$,  using that $\sum d_i=6(2k-1)$, 
$d_3+d_4\le 6(2k-1)-(2k-3)-5(2k-3)=12$, so from $d_4\le d_3$ we get $d_4\le 6$ and from $d_4\ge 5$ we get $d_3\le 7$.
If $d_4=5$ then $d_5=5$, otherwise we have that $d_4=d_3=6$ and
$d_5\le 6(2k-1)-(2k-3)-12-5(2k-4)=5$. 
\end{proof}

\begin{lemma}\label{la:mvsodd}
Let $n\geq 15$ be odd, $k=\lfloor\frac{n}{2}\rfloor$, and let $G$ be a $5$-connected triangulation of order $n$, with $W(G)\le W(T_n^5)$. Let 
$v$ be a vertex with $d(v)=\Delta(G)=:\Delta$.
 Consider the extended mosaic graph $M_v^{\star}$, and the sets $W=\{w_1,\ldots,w_{\Delta}\}$ and $U=\{u_1,\ldots,u_{\Delta}\}$. Let $x_1,x_2$ denote  the
 two  vertices not in
 $M_v^{\star}$.
 The following statements are true:
  \begin{enumerate}[label={\upshape (\alph*)}]
  \item\label{case:5degs}  $\Delta=k -1$, at most $4$ degrees of $G$ are larger than $5$, and for the largest $4$ degrees $\Delta\ge d_2\ge d_3\ge d_4$ of $G$ we have either
  $d_2\ge k-2$, $d_3\le 7$, $d_4\le 6$ and $d_3+d_4\le 12$, or $n=23$, $d_2=d_3=8=k-3$, $d_4=5$.
  \item\label{case:paths} For $i\in \{1,2\}$, there are vertices $a_i,b_i\in V(C_v)$ such that
$N(x_i)\setminus\{x_{3-i}\}=V(P_v(a_i,b_i))$. (We refer to these  $a_i,b_i$ vertices in the forthcoming claims.) Furthermore, if $c\in V(P_v(a_1,b_1))\cap V(P_v(a_2,b_2))$,
 then $c=a_1=b_2$ or $c=a_2=b_1$.
\item\label{case:pathinw} 
$C_v=w_1w_2\ldots w_{\Delta}$,
$d(x_i)\le \Delta-1$, and if $x_1x_2\notin E(G)$ then $d(x_i)\le \Delta-2$ and $n\ge 19$. 
\item\label{case:xy} If $x_1x_2\in E(G)$, then $G\simeq T_n^5$
\item\label{case:noxy} If $x_1x_2\notin E(G)$ then $a_1b_1,a_2b_2\in E(G)\setminus E(M_v^{\star})$. Moreover, for $i\in\{1,2\}$,
if  $c_i\in V(P_v(b_i,a_{3-i}))$ and $z\in V(G)\setminus\{x_1,x_2\}$, such that
$c_iz\in E(G)\setminus E(M_v^{\star})$, then $z\in V(P_v(b_{3-i},a_i))$,
and the neighbors of $c_i$ in $P_v(b_{3-i},a_i)$ form a consecutive sequence of vertices on this path. 
\item\label{case:allnoteq}
If $x_1x_2\notin E(G)$, then $a_1= b_2$ or $a_2= b_1$.
\item\label{case:twonoteq}  If $x_1x_2\notin E(G)$, then $a_1=b_2$ and $a_2= b_1$.
\item\label{case:alleq} If $x_1x_2\notin E(G)$, then $G\simeq X$ and $W(G)=W(T_{19}^5)$.
\end{enumerate}
\end{lemma}

\begin{proof} Note that $n\ge 15$, so $k\ge 7$. Let $G$, $v$, $x_1,x_2$ be as in the conditions.  Lemma~\ref{la:mvsdeg} yields \ref{case:5degs} . 

\ref{case:paths}: Assume  $i\in \{1,2\}$.
As $d(x_i)\ge 5$, $x_i$ has at least $4$ neighbors on $C_v$, so there are vertices $a_i,b_i\in N(x_i)\cap V(C_v)$ such that
all vertices in $N(x_i)\setminus \{x_{3-i} \} $ lie on the path $P_v(a_i,b_i)$ and $x_{3-i}$ does not lie in the interior of the subregion
of ${\mathcal{Q}}_v^{\star}$ bounded by the cycle $x_iP_v(a_i,b_i)$. 
By Lemma~\ref{la:mvs}~\ref{case:cutedge}, no two vertices in $P_v(a_i,b_i)$ can be joined by an edge that is not in $M_v^{\star}$. As every vertex of $C_v$
has at least one edge incident upon it from $E(G)\setminus E(M_v^{\star})$, we must have $V(P_v(a_i,b_i))\subseteq N(x_i)$, therefore
$V(P_v(a_i,b_i))=N(x_i)\setminus \{ x_{3-i} \} $. As all edges incident upon $x_1$ or $x_2$ lie in the region ${\mathcal{Q}}_v^{\star}$ and do not cross, the two paths share at most their endvertices. 

\ref{case:pathinw}: Assume to the contrary that $c\in U\cap V(P_v(a_i,b_i))$. As $d(x_i)\ge 5$, $P_v(a_i,b_i)$ has at least $4$ vertices. Therefore,
there exists an internal  vertex $c^{\star}$ of the path $P_v(a_i,b_i)$,
such that the edge $cc^{\star}$ is an edge of this path. This means 
$c^{\star}\in W$ (see Lemma~\ref{la:mvs}~\ref{case:cutedge}), $c^{\star}$ has at most $3$ edges incident upon it in $E(M_v^{\star})$, and
the only edge in $E(G)\setminus E(M_v^{\star})$ incident upon $c^{\star}$ is $c^{\star}x_i$, contradicting $d(c^{\star})\ge 5$. Thus,
$V(P_v(a_i,b_i))\subseteq W$. By Lemma~\ref{la:mvs}~\ref{case:edgesout} the internal vertices of the paths $P(b_i,a_{3-i})$ each have at least one edge
of $E(G)\setminus E(M_v^{\star})$ incident upon them. Thus, by the definition of $a_i,b_i$ and  $M_v^{\star}$, the internal vertices of the paths $P(b_i,a_{3-i})$ have to be adjacent to at least one
other internal vertex of the paths $P(b_i,a_{3-i})$. Lemma~\ref{la:mvs}~\ref{case:cutedge} implies that $V(P(b_i,a_{3-i}))\subseteq W$. So $V(C_v)=W$.
By part \ref{case:paths},  we have $|V(P_v(a_{3-i},b_{3-i})) \setminus V(P_v(a_i,b_i))|\geq 2$. Hence
$\Delta=|W| \geq |V(P_v(a_i,b_i))|+|V(P_v(a_{3-i},b_{3-i})) \setminus V(P_v(a_i,b_i))|$. As $N(x_{i})=V(P_v(a_i,b_i))$ or $V(P_v(a_i,b_i)) \cup \{x_{3-i} \}  $, depending 
on whether $x_1x_2$ is non-edge or edge,   the claimed upper bounds on $d(x_i)$ follow. Assume $x_1x_2\notin E(G)$. As we have
$5\le d(x_i)\le\Delta -2$, we have $\Delta\ge 7$, so $n\ge 17$. If $n=17$, then we must have $d(x_1)=d(x_2)=5$. As
$8=d(x_1)+d(x_2)-2\le |W|=7$, this is a contradiction. \ref{case:pathinw} follows.

 \ref{case:xy}: Let $x_1x_2\in E(G)$. Assume that $b_2\not=a_1$. Then  either $P_v(b_2,a_1)$ has an internal vertex $c$ or $b_2a_1\in E(G)$. 
 In the first case, as $c\in C_v$, there is at least one edge in 
$E(G)\setminus E(M_v^{\star})$ incident upon $c$ by Lemma~\ref{la:mvs}~\ref{case:edgesout}.
All edges of $E(G)\setminus E(M_v^{\star})$ incident upon
$c$ must lie in the subregion of ${\mathcal{Q}}_v^{\star}$ bounded by the cycle $P_v(b_2,a_1)x_1x_2$. By Lemma~\ref{la:mvs}~\ref{case:cutedge}
these edges must be of the form $cx_1$ or $cx_2$. But none of these are edges of $G$, which is a contradiction. 
    In the second case  $ b_2a_1\in E(G)$, and 
 the subregion of ${\mathcal{Q}}_v^{\star}$ bounded by the $4$-cycle $b_2a_1x_1x_2$ has no vertices in its interior, so
 we must have either $x_1b_2\in E(G)$ or $x_2a_1\in E(G)$, a contradiction. So $a_1=b_2$, and $a_2=b_1$, 
 and 
  $V(C_v)=W$  by \ref{case:pathinw}. All $w\in W$ are incident to 4 edges of $M_v^{\star}$, but $a_1,a_2\in W$ are incident to 2 more edges, and vertices
  of $W\setminus \{a_1,a_2\}$ are incident to one more, so $d(a_1)=d(a_2)=6$, and \ref{case:5degs}
gives $d_2\ge k-2$. Since $a_1,a_2$ and $v$ are vertices with degree greater than $5$,
and $G$ has at most $4$ vertices with degree greater than $5$, 
 we get $\min(d(x_1),d(x_2))=5$, which gives that $G\simeq T_n^5$ as claimed.

For the remaining cases assume that $x_1$ and $x_2$ are not adjacent, so $n\ge 19$ and $\Delta\ge 8$.
This also implies that $N(x_i)=V(P(a_i,b_i))$, so the paths $P(a_i,b_i)$ have at least $5$ vertices.

\ref{case:noxy}: 
In this case the edges $x_ia_i, x_ib_i$  lie on the boundary of the same face, so
$a_ib_i\in E(G)$, and $a_ib_ix_i$ is a boundary of a face.  Moreover, as $|V(P(a_i,b_i)|\ge 5$, $a_ib_i\notin E(M_v^{\star})$.
The rest of the statement is trivial if $a_1=b_2$ and $a_2=b_1$, so assume that is not the case. 
Consider the  connected subregion ${\mathcal{R}}$ of ${\mathcal{Q}}_v^{\star}$ bounded by the  cycle $P(b_1,a_2)P(b_2,a_1)$ (that has length at least $3$ by the assumption); it has no vertices in its interior.
Any edges between vertices of the cycle $P(b_1,a_2)P(b_2,a_1)$ are edges of this cycle or lie inside $\mathcal{R}$.
This finishes the proof  unless $a_1\ne b_2$ and $a_2\ne b_1$, so consider that to be the case.
Let  $c_i\in V(P_v(b_i,a_{3-i}))$ and $z\in V(G)\setminus\{x_1,x_2\}$ such that
$c_iz\in E(G)\setminus E(M_v^{\star})$.
As $c_i$ lies on the boundary of the connected subregion ${\mathcal{R}}$,
Lemma~\ref{la:mvs}~\ref{case:cutedge} gives that
$z\in V(P(b_{3-i},a_i))$, as claimed. Also, if $z_1,z_2 \in V(P(b_{3-i},a_i))$ are different neighbors of $c_i$ where $z_1z_2$ is not an edge of the path
$P_v(b_{3-i},a_i)$, 
then by Lemma~\ref{la:mvs}~\ref{case:cutedge} any internal vertex $z_3$
of the $z_1-z_2$ subpath of   $P(b_{3-i},a_i)$
can only have the edge $z_3c_i$ incident upon it from $E(G)\setminus E(M_v^{\star})$. Since by Lemma~\ref{la:mvs}~\ref{case:edgesout} $z_3$ must have an edge from $E(G)\setminus E(M_v^{\star})$
incident upon it, \ref{case:noxy} follows.

\ref{case:allnoteq}: 
Assume to the contrary that $a_1\ne b_2$ and $a_2\ne b_1$. By~\ref{case:noxy} we have $a_1b_1, a_2b_2\in E(G)\setminus E(M_v^{\star})$.
Let ${\mathcal{R}}$ be the connected subregion of ${\mathcal{Q}}_v^{\star}$ bounded by the cycle $P_v(b_2,a_1)P_v(b_1,a_2)$.
$G$ has at most 4 vertices with degree greater than 5.
As $d(v)=\Delta>6$, $V(G)\setminus\{v\}$ has at most $3$ vertices
with degree greater than $5$.
In particular, $C_v$ contains at most $3$ vertices with degree greater than $5$.
As by \ref{case:pathinw} 
$V(C_v)=W$,
each of
 $a_1,a_2,b_1,b_2$ has at least $4$ edges incident upon them in $E(M_v^{\star})$, and two edges incident upon them from $E(G)\setminus E(M_v^{\star})$ 
 (the edges $a_ib_i$, $a_ix_i$, $b_ix_i$). This gives that $a_1,b_1,a_2,b_2$ have degree at least $6$, a contradiction.
 \ref{case:allnoteq} follows.

	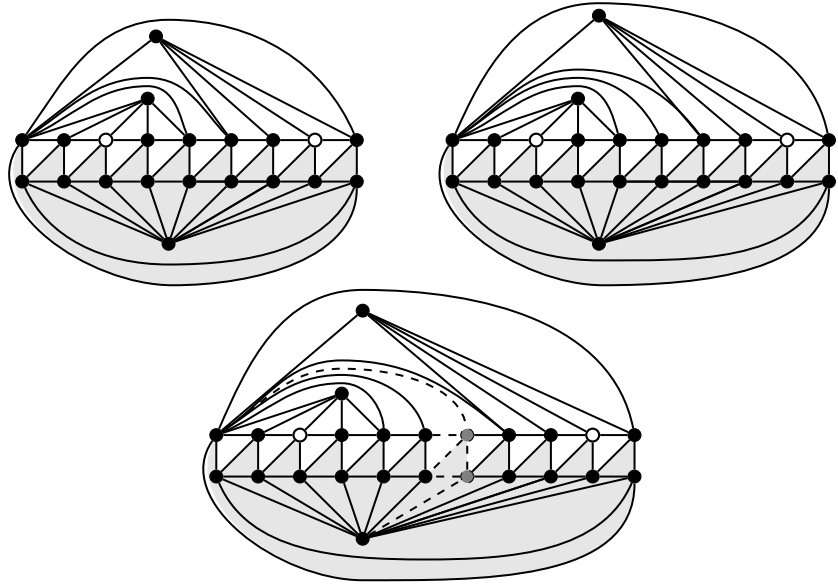
\begin{figure}[htbp]
		\centering
			\begin{tikzpicture}
			[scale=0.55,yscale=-1,inner sep=0.6mm,vertex/.style={circle,draw},thickedge/.style={line width=0.75pt}]
			\coordinate (A1) at (0,3);
			\coordinate (A2) at (1,3);
			\coordinate (A3) at (2,3);
			\coordinate (A4) at (3,3);
			\coordinate(A6) at (5,3) ;			
			\coordinate (A5) at (4,3);
			\coordinate (A7) at (6,3) ;
			\coordinate (A8) at (7,3);
			\coordinate (A9) at (8,3);								
			\coordinate (B1) at (0,2) ;
			\coordinate (B4) at (3,2);
			\coordinate (B6) at (5,2) ;			
			\coordinate (B5) at (4,2) ;
			\coordinate (B7) at (6,2) ;
			\coordinate (B8) at (7,2) ;
			\coordinate (B9) at (8,2);		
			\coordinate (B2) at (1,2) ;
			\coordinate (B3) at (2,2);
			
            \fill[gray!20] (B1) to[out=-240,in=180] (3.6,5.5) to[out=0,in=-270] (A9)--(B9)--(A8)--(B8)--(A7)--(B7)--(A6)--(B6)--(A5)--(B5)--(A4)--(B4)--(A3)--(B3)--(A2)--(B2)--(A1)--(B1);				
			
			\node[vertex] (p1) at (3.5,4.5) [fill=black] {};
			\node[vertex] (p2) at (3.2,-.5) [fill=black] {};		
			\node[vertex] (a1) at (0,3) [fill=black] {};
			\node[vertex] (a2) at (1,3) [fill=black] {};
			\node[vertex] (a3) at (2,3) [fill=black] {};
			\node[vertex] (a4) at (3,3) [fill=black] {};
			\node[vertex] (a9) at (5,3) [fill=black] {};			
			\node[vertex] (a5) at (4,3) [fill=black] {};
			\node[vertex] (a6) at (6,3) [fill=black] {};
			\node[vertex] (a7) at (7,3) [fill=black] {};
			\node[vertex] (a8) at (8,3) [fill=black] {};								
			\node[vertex] (b1) at (0,2) [fill=black] {};
			\node[vertex] (b2) at (3,1) [fill=black] {}; 
			\node[vertex] (b3) at (3,2) [fill=black] {};
			\node[vertex] (b8) at (5,2) [fill=black] {};			
			\node[vertex] (b4) at (4,2) [fill=black] {};
			\node[vertex] (b5) at (6,2) [fill=black] {};
			\node[vertex,thick] (b6) at (7,2) [fill=white] {};
			\node[vertex] (b7) at (8,2) [fill=black] {};		
			\node[vertex] (d1) at (1,2) [fill=black] {};
			\node[vertex,thick] (d2) at (2,2) [fill=white] {};
			\draw[thickedge] (p1)--(a1);
			\draw[thickedge] (p1)--(a2);
			\draw[thickedge] (p1)--(a3);
			\draw[thickedge] (p1)--(a4);
			\draw[thickedge] (p1)--(a8);
			\draw[thickedge] (p1)--(a9);			
            \draw[thickedge] (p1)--(a5);
            \draw[thickedge] (p1)--(a6);
            \draw[thickedge] (p1)--(a7);
			\draw[thickedge] (a1)--(a2)--(a3)--(a4);
			\draw[thickedge] (a4)--(a5);
			\draw[thickedge] (a5)--(a6);
			\draw[thickedge] (a6)--(a7);
			\draw[thickedge] (a5)--(a6);
			\draw[thickedge] (p1)--(a6);
			\draw[thickedge] (p2)--(b1);
			\draw[thickedge] (b4)--(b2); 
			\draw[thickedge] (b1) to[out=330,in=180] (3,.7) to[out=0,in=255] (b4);   	
			\draw[thickedge] (b1) to[out=320,in=180] (3,.5) to[out=0,in=235] (b8);   	
			
			\draw[thickedge] (p2)--(b8); %
			\draw[thickedge] (p2)--(b5);
			\draw[thickedge] (p2)--(b6);
			\draw[thickedge] (p2)--(b7);
			\draw[thickedge] (b1)--(b2)--(b3)--(b4);
			\draw[thickedge] (b4)--(b5);
			\draw[thickedge] (b5)--(b6);
			\draw[thickedge] (b6)--(b7);
			\draw[thickedge] (a1) to[out=70,in=180] (3.5,5) to[out=0,in=110] (a8);
            \draw[thickedge] (b1) to[out=300,in=180] (3.5,-.9) to[out=0,in=250] (b7);
            \draw[thickedge] (b1) to[out=-240,in=180] (3.6,5.5) to[out=0,in=-270] (a8);	
            \draw[thickedge] (d1)--(a1);
			\draw[thickedge] (d1)--(a2);
			\draw[thickedge] (d1)--(b1);
			\draw[thickedge] (d1)--(b2);
			\draw[thickedge] (d1)--(d2);
			\draw[thickedge] (d2)--(a2);
			\draw[thickedge] (d2)--(a3);
			\draw[thickedge] (d2)--(b2); 
			\draw[thickedge] (d2)--(b3);
			\draw[thickedge] (a1)--(b1);
			\draw[thickedge] (a3)--(b3);
			\draw[thickedge] (a4)--(b4);
			\draw[thickedge] (a5)--(b4);%
			\draw[thickedge] (a4)--(b3);
			\draw[thickedge] (a5)--(b8)--(a9)--(b5);%
			\draw[thickedge] (a6)--(b5);
			\draw[thickedge] (a6)--(b6);
			\draw[thickedge] (a7)--(b6);
			\draw[thickedge] (a8)--(b7);
			\draw[thickedge] (a8)--(a7);
			\draw[thickedge] (a7)--(b7);	
			\end{tikzpicture}
			\quad
			\begin{tikzpicture}
			[scale=0.55,yscale=-1,inner sep=0.6mm,vertex/.style={circle,draw},thickedge/.style={line width=0.75pt}]
			\coordinate (A1) at (0,3);
			\coordinate (A2) at (1,3);
			\coordinate (A3) at (2,3);
			\coordinate (A4) at (3,3);
			\coordinate(A6) at (5,3) ;			
			\coordinate (A5) at (4,3);
			\coordinate (A7) at (6,3) ;
			\coordinate (A8) at (7,3);
			\coordinate (A9) at (8,3);	
			\coordinate (A10) at (9,3);																
			\coordinate (B1) at (0,2) ;
			\coordinate (B4) at (3,2);
			\coordinate (B6) at (5,2) ;			
			\coordinate (B5) at (4,2) ;
			\coordinate (B7) at (6,2) ;
			\coordinate (B8) at (7,2) ;
			\coordinate (B9) at (8,2);		
			\coordinate (B2) at (1,2) ;
			\coordinate (B3) at (2,2);
			\coordinate (B10) at (9,2) ;
			
            \fill[gray!20] (B1) to[out=-240,in=180] (3.6,5.5) to[out=0,in=-270] (A10)--(B10)--(A9)--(B9)--(A8)--(B8)--(A7)--(B7)--(A6)--(B6)--(A5)--(B5)--(A4)--(B4)--(A3)--(B3)--(A2)--(B2)--(A1)--(B1);							
			
			\node[vertex] (p1) at (3.5,4.5) [fill=black] {};
			\node[vertex] (p2) at (3.5,-1) [fill=black] {};
			\node[vertex] (a1) at (0,3) [fill=black] {};
			\node[vertex] (a2) at (1,3) [fill=black] {};
			\node[vertex] (a3) at (2,3) [fill=black] {};
			\node[vertex] (a4) at (3,3) [fill=black] {};
			\node[vertex] (a10) at (5,3) [fill=black] {};						
			\node[vertex] (a9) at (6,3) [fill=black] {};			
			\node[vertex] (a5) at (4,3) [fill=black] {};
			\node[vertex] (a6) at (7,3) [fill=black] {};
			\node[vertex] (a7) at (8,3) [fill=black] {};
			\node[vertex] (a8) at (9,3) [fill=black] {};								
			\node[vertex] (b1) at (0,2) [fill=black] {};
			\node[vertex] (b2) at (3,1) [fill=black] {}; 
			\node[vertex] (b3) at (3,2) [fill=black] {};
			\node[vertex] (b9) at (5,2) [fill=black] {};						
			\node[vertex] (b8) at (6,2) [fill=black] {};			
			\node[vertex] (b4) at (4,2) [fill=black] {};
			\node[vertex] (b5) at (7,2) [fill=black] {};
			\node[vertex,thick] (b6) at (8,2) [fill=white] {};
			\node[vertex] (b7) at (9,2) [fill=black] {};		
			\node[vertex] (d1) at (1,2) [fill=black] {};
			\node[vertex,thick] (d2) at (2,2) [fill=white] {};
			\draw[thickedge] (p1)--(a1);
			\draw[thickedge] (p1)--(a2);
			\draw[thickedge] (p1)--(a3);
			\draw[thickedge] (p1)--(a4);
			\draw[thickedge] (p1)--(a8);
			\draw[thickedge] (p1)--(a9);	
			\draw[thickedge] (p1)--(a10);						
            \draw[thickedge] (p1)--(a5);
            \draw[thickedge] (p1)--(a6);
            \draw[thickedge] (p1)--(a7);
			\draw[thickedge] (a1)--(a2)--(a3)--(a4);
			\draw[thickedge] (a4)--(a5);
			\draw[thickedge] (a5)--(a6);
			\draw[thickedge] (a6)--(a7);
			\draw[thickedge] (a5)--(a6);
			\draw[thickedge] (p1)--(a6);
			\draw[thickedge] (p2)--(b1);
			\draw[thickedge] (b4)--(b2); 
			\draw[thickedge] (b1) to[out=330,in=180] (3,.7) to[out=0,in=260] (b4);   	
			\draw[thickedge] (b1) to[out=320,in=180] (3,.3) to[out=0,in=240] (b8);   	
			\draw[thickedge] (b1) to[out=320,in=180] (3,.5) to[out=0,in=250] (b9);   				
			\draw[thickedge] (p2)--(b8); %
			\draw[thickedge] (p2)--(b5);
			\draw[thickedge] (p2)--(b6);
			\draw[thickedge] (p2)--(b7);
			\draw[thickedge] (b1)--(b2)--(b3)--(b4);
			\draw[thickedge] (b4)--(b5);
			\draw[thickedge] (b5)--(b6);
			\draw[thickedge] (b6)--(b7);
			\draw[thickedge] (a1) to[out=70,in=180] (3.5,4.9) to[out=0,in=110] (a8);
            \draw[thickedge] (b1) to[out=290,in=180] (3.5,-1.3) to[out=0,in=260] (b7);
             \draw[thickedge] (b1) to[out=-240,in=180] (3.6,5.5) to[out=360,in=-270] (a8);	
            \draw[thickedge] (d1)--(a1);
			\draw[thickedge] (d1)--(a2);
			\draw[thickedge] (d1)--(b1);
			\draw[thickedge] (d1)--(b2);
			\draw[thickedge] (d1)--(d2);
			\draw[thickedge] (d2)--(a2);
			\draw[thickedge] (d2)--(a3);
			\draw[thickedge] (d2)--(b2); 
			\draw[thickedge] (d2)--(b3);
			\draw[thickedge] (a1)--(b1);
			\draw[thickedge] (a3)--(b3);
			\draw[thickedge] (a4)--(b4);
			\draw[thickedge] (a5)--(b4);
			\draw[thickedge] (a4)--(b3);
			\draw[thickedge] (a10)--(b8)--(a9)--(b5);
			\draw[thickedge] (a10)--(b9);%
			\draw[thickedge] (a5)--(b9); %
			\draw[thickedge] (a6)--(b5);
			\draw[thickedge] (a6)--(b6);
			\draw[thickedge] (a7)--(b6);
			\draw[thickedge] (a8)--(b7);
			\draw[thickedge] (a8)--(a7);
			\draw[thickedge] (a7)--(b7);	
			\end{tikzpicture}
			
			\begin{tikzpicture}
			[scale=0.55,yscale=-1,inner sep=0.6mm,vertex/.style={circle,draw},thickedge/.style={line width=0.75pt}]
			
			\coordinate (A1) at (0,3);
			\coordinate (A2) at (1,3);
			\coordinate (A3) at (2,3);
			\coordinate (A4) at (3,3);
			\coordinate(A6) at (5,3) ;			
			\coordinate (A5) at (4,3);
			\coordinate (A7) at (6,3) ;
			\coordinate (A8) at (7,3);
			\coordinate (A9) at (8,3);	
			\coordinate (A10) at (9,3);
			\coordinate (A11) at (10,3);																	
			\coordinate (B1) at (0,2) ;
			\coordinate (B4) at (3,2);
			\coordinate (B6) at (5,2) ;			
			\coordinate (B5) at (4,2) ;
			\coordinate (B7) at (6,2) ;
			\coordinate (B8) at (7,2) ;
			\coordinate (B9) at (8,2);		
			\coordinate (B2) at (1,2) ;
			\coordinate (B3) at (2,2);
			\coordinate (B10) at (9,2) ;
			\coordinate (B11) at (10,2);
			
            \fill[gray!20] (B1) to[out=-240,in=180] (3.6,5.5) to[out=0,in=-270] (A11)--(B11)--(A10)--(B10)--(A9)--(B9)--(A8)--(B8)--(A7)--(B7)--(A6)--(B6)--(A5)--(B5)--(A4)--(B4)--(A3)--(B3)--(A2)--(B2)--(A1)--(B1);

			\node[vertex] (p1) at (3.5,4.5) [fill=black] {};
			\node[vertex] (p2) at (3.5,-1) [fill=black] {};
			\node[vertex] (a1) at (0,3) [fill=black] {};
			\node[vertex] (a2) at (1,3) [fill=black] {};
			\node[vertex] (a3) at (2,3) [fill=black] {};
			\node[vertex] (a4) at (3,3) [fill=black] {};
			\node[vertex] (a11) at (5,3) [fill] {};							
			\node[vertex,dashed] (a10) at (6,3) [fill=gray] {};						
			\node[vertex] (a9) at (7,3) [fill=black] {};			
			\node[vertex] (a5) at (4,3) [fill=black] {};
			\node[vertex] (a6) at (8,3) [fill=black] {};
			\node[vertex] (a7) at (9,3) [fill=black] {};
			\node[vertex] (a8) at (10,3) [fill=black] {};								
			\node[vertex] (b1) at (0,2) [fill=black] {};
			\node[vertex] (b2) at (3,1) [fill=black] {}; 
			\node[vertex] (b3) at (3,2) [fill=black] {};
			\node[vertex,dashed] (b9) at (6,2) [fill=gray] {};						
			\node[vertex] (b8) at (7,2) [fill=black] {};			
			\node[vertex] (b4) at (4,2) [fill=black] {};
			\node[vertex] (b5) at (8,2) [fill=black] {};
			\node[vertex,thick] (b6) at (9,2) [fill=white] {};
			\node[vertex] (b7) at (10,2) [fill=black] {};	
			\node[vertex] (b10) at (5,2) [fill=black] {};						
			\node[vertex] (d1) at (1,2) [fill=black] {};
			\node[vertex,thick] (d2) at (2,2) [fill=white] {};
			\draw[thickedge] (p1)--(a1);
			\draw[thickedge] (p1)--(a2);
			\draw[thickedge] (p1)--(a3);
			\draw[thickedge] (p1)--(a4);
			\draw[thickedge] (p1)--(a8);
			\draw[thickedge] (p1)--(a9);	
			\draw[thickedge,dashed] (p1)--(a10);	
			\draw[thickedge] (p1)--(a11);														
            \draw[thickedge] (p1)--(a5);
            \draw[thickedge] (p1)--(a6);
            \draw[thickedge] (p1)--(a7);
			\draw[thickedge] (a1)--(a2)--(a3)--(a4);
			\draw[thickedge] (a4)--(a5);
			\draw[thickedge] (a5)--(a11);
			\draw[thickedge] (a10)--(a6);
			\draw[thickedge,dashed] (a10)--(a11);			
			\draw[thickedge] (a6)--(a7);
			\draw[thickedge] (p1)--(a6);
			\draw[thickedge] (p2)--(b1);
			\draw[thickedge] (b4)--(b2); 
			\draw[thickedge] (b1) to[out=330,in=180] (3,.75) to[out=0,in=270] (b4);   	
			\draw[thickedge] (b1) to[out=320,in=180] (3,.2) to[out=0,in=225] (b8);   	
			\draw[thickedge,dashed] (b1) to[out=320,in=180] (3,.4) to[out=0,in=265] (b9);   	
			\draw[thickedge] (b1) to[out=320,in=180] (3,.55) to[out=0,in=260] (b10);   								
			\draw[thickedge] (p2)--(b8); %
			\draw[thickedge] (p2)--(b5);
			\draw[thickedge] (p2)--(b6);
			\draw[thickedge] (p2)--(b7);
			\draw[thickedge] (b1)--(b2)--(b3)--(b4);
			\draw[thickedge] (b4)--(b10);
			\draw[thickedge,dashed] (b10)--(b9)--(a11);
			\draw[thickedge] (b9)--(b5);
			\draw[thickedge] (b5)--(b6);
			\draw[thickedge] (b6)--(b7);
			\draw[thickedge] (a1) to[out=70,in=180] (5,5) to[out=0,in=110] (a8);
            \draw[thickedge] (b1) to[out=290,in=180] (3.5,-1.5) to[out=0,in=260] (b7);
             \draw[thickedge] (b1) to[out=-240,in=180] (3.5,5.5) to[out=360,in=-270] (a8);	
            \draw[thickedge] (d1)--(a1);
			\draw[thickedge] (d1)--(a2);
			\draw[thickedge] (d1)--(b1);
			\draw[thickedge] (d1)--(b2);
			\draw[thickedge] (d1)--(d2);
			\draw[thickedge] (d2)--(a2);
			\draw[thickedge] (d2)--(a3);
			\draw[thickedge] (d2)--(b2); 
			\draw[thickedge] (d2)--(b3);
			\draw[thickedge] (a1)--(b1);
			\draw[thickedge] (a3)--(b3);
			\draw[thickedge] (a4)--(b4);
			\draw[thickedge] (a5)--(b4);
			\draw[thickedge] (a4)--(b3);
			\draw[thickedge] (a10)--(b8)--(a9)--(b5);
			\draw[thickedge,dashed] (a10)--(b9);%
			\draw[thickedge] (a11)--(b10);%
			
			\draw[thickedge] (a5)--(b10); %
			\draw[thickedge] (a6)--(b5);
			\draw[thickedge] (a6)--(b6);
			\draw[thickedge] (a7)--(b6);
			\draw[thickedge] (a8)--(b7);
			\draw[thickedge] (a8)--(a7);
			\draw[thickedge] (a7)--(b7);	
			\end{tikzpicture}			
			\caption{5-connected triangulations of order  $n=21$, $23$ and $n\geq 25$, which have the same degree sequence as $T_n^5$. 
			The gray regions show the mosaic graphs around the vertex of degree $k-1$. The gray vertices and dashed edges on the triangulation of order $25$  indicate the pattern to be repeated to get   the construction for higher odd order. 
			The two white vertices are at distance $4$.} 
			\label{fig:t5cotherodd2}
    \end{figure}

\ref{case:twonoteq}: 
Assume to the contrary that $a_1\ne b_2$ or $a_2\ne b_1$. By~\ref{case:allnoteq}, without loss of generality we have
$a_1=b_2$ and  $a_2\ne b_1$.  By definition $a_1x_1,a_1x_2\in E(G)$ and by~\ref{case:noxy} all vertices of $P_v(b_1,a_2)$ are neighbors
of $a_1$. All other neighbors of $a_1$ are one of the 4 neighbors of $a_1$ in $M_v^{\star}$. Consequently $d(a_1)=6+|V(P_v(b_1,a_2))|\ge 8$,
so $d(a_1)=d_2\in\{k-1,k-2,k-3\}$, and if $d(a_1)=k-3$, then $G$ contains no degree $6$ vertices. 
As $V(C_v)=W$, we must have $d(a_2)=d(b_1)=6$, as $a_2,b_1$ each have 4 neighbors in $M_v^{\star}$, and both
are joined to $a_1=b_2$, and to a single $x_i$, and not joined to anything else. By~\ref{case:5degs} $d(a_1)\ge k-2$, and,
as $v,a_1,a_2,b_1$ are the $4$ vertices of degree greater than $5$, all other vertices (including $x_1$ and $x_2$) have degree $5$.
So every $w\in W\setminus \{a_1,a_2,b_1\}$  has  4 neighbors in $M_v^{\star}$, and is joined by an edge in $E(G)\setminus E(M_v^{\star})$  
to exactly one of the vertices $x_1,a_1,x_2$, and  the paths $P(a_i,b_i)$ have $5$ vertices each.
As the sum of degrees is $6n-12=k-1+d(a_1)+12+5(n-4)$, we get $d(a_1)=k-2$. As $d(a_1)\ge 8$, this gives $n\ge 21$. Figure~\ref{fig:t5cotherodd2} has the graph $G$ for all $n\ge 21$.
Since
 $G$ has the same degree sequence
as $T_n^5$ and $W(G)\le W(T_n^5)$, by Lemma~\ref{la:basics}~\ref{case:lowerbound} we must have $W(G)=W(T_n^5)$ and the diameter of $G$ is at most $3$. However, $G$ has diameter at least $4$, as demonstratred on Figure~\ref{fig:t5cotherodd2}, a contradiction.
\ref{case:twonoteq} follows.

\ref{case:alleq}: By \ref{case:twonoteq}, $a_1=b_2$ and  $a_2=b_1$. By \ref{case:noxy} $a_1a_2\in E(G)$.
By \ref{case:pathinw} $V(C_v)=W$, and any edge from $E(G)\setminus E(M_v^{\star})$ incident upon a vertex 
$w\in W\setminus\{a_1,a_2\}$ connects $w$ to exactly one of $x_1,x_2$.  
Each  $a_i$ has 4 incident edges in $E(M_v^{\star})$, and in addition, it is joined to exactly 3 more vertices: $x_1,x_2,a_{3-i}$. So $d(a_1)=d(a_2)=7=d_2=d_3$. By \ref{case:5degs}, 
$d_3+d_4\le 12$, consequently all vertices of $V(G)\setminus\{v,a_1,a_2\}$ (including $x_1$ and $x_2$) have degree $5$. As
$6n-12=k-1+14+5(n-3)$, $n=19$. We have that $G\simeq X$ and $W(G)=W(T_{19}^5)$ (see Figure~\ref{fig:t5cotherodd1}).
\end{proof}

The following theorem now follows:
\begin{theorem} 
\label{t5cminwiodd}
Let $n\ge 15$ be odd. If $n\ne 19$, then the unique minimizer of the  Wiener index    among $5$-connected triangulations of order $n$  is $T_n^5$. If $n=19$, then there are precisely two minimizers, $T_{19}^5$ and $X$.
\end{theorem}


\bibliographystyle{amcjoucc}

\end{document}